\documentclass{amsart}
\usepackage{graphicx} 

\usepackage{amsthm}
\usepackage{amsmath}
\usepackage{amssymb}
\usepackage{thmtools}
\usepackage{thm-restate}
\usepackage[backend=biber,
style=alphabetic,
]{biblatex}
\DeclareFieldFormat{postnote}{#1}
\DeclareFieldFormat{multipostnote}{#1}
\usepackage{tikz}
\usepackage{circuitikz}
\usepackage{hyperref}
\usepackage{subfiles}
\usepackage{caption}
\DeclareCaptionFormat{custom}
{%
    \textbf{#1#2}\textit{\small #3}
}
\theoremstyle{definition}
\newtheorem{counter}{bad bad bad}[section]
\newtheorem{lemma}[counter]{Lemma}
\newtheorem{definition}[counter]{Definition}
\newtheorem{notation}[counter]{Notation}
\newtheorem{remark}[counter]{Remark}
\newtheorem{proposition}[counter]{Proposition}

\newtheorem{fact}[counter]{Fact}
\newtheorem{corollary}[counter]{Corollary}
\newtheorem{hypothesis}[counter]{Hypothesis}
\newtheorem{question}[counter]{Question}
\newtheorem{example}[counter]{Example}

\newtheorem{innercustomthm}{Theorem}
\newenvironment{customthm}[1]
  {\renewcommand\theinnercustomthm{#1}\innercustomthm}
  {\endinnercustomthm}

\newcommand{\K}{\mathbf{K}}
\newcommand{\gS}{\operatorname{gS}}
\newcommand{\LS}{\operatorname{LS}}
\newcommand{\gtp}{\mathbf{gtp}}

\newcommand{\lek}{\leq_{\K}}

\newcommand{\underhlim}{\leq^{(\lambda, \geq \kappa)}_{\K}}
\newcommand{\cof}{\text{cf}}
\newcommand{\calt}{\mathcal{T}}
\newcommand{\calc}{\mathcal{C}}
\newcommand{\lesst}{\vartriangleleft}
\newcommand{\lesstrong}{\lesst_{(\lambda, \geq \kappa)}}

\newcommand{\kkappalims}{K_{(\lambda, \geq \kappa)}}
\newcommand{\Kkappalims}{\K_{(\lambda, \geq \kappa)}}

\newbox\noforkbox \newdimen\forklinewidth
\setbox0\hbox{$\textstyle\smile$}
\setbox1\hbox to \wd0{\hfil\vrule width \forklinewidth depth-2pt
 height 10pt \hfil}
\setbox\noforkbox\hbox{\lower 2pt\box1\lower 2pt\box0\relax}
\def\unionstick{\mathop{\copy\noforkbox}\limits}

\def\nonfork_#1{\unionstick_{\textstyle #1}}

\setbox0\hbox{$\textstyle\smile$}
\setbox1\hbox to \wd0{\hfil{\sl /\/}\hfil}
\setbox2\hbox to \wd0{\hfil\vrule height 10pt depth -2pt width
              \forklinewidth\hfil}
\newbox\doesforkbox
\setbox\doesforkbox\hbox{\lower 2pt\box1 \lower 2pt\box2\lower2pt\box0\relax}
\def\nunionstick{\mathop{\copy\doesforkbox}\limits}

\def\fork_#1{\nunionstick_{\textstyle #1}}

\newcommand{\dnf}{\unionstick}

\newcommand{\nf}{\unionstick}
\newcommand{\dnfb}[4]{#2 \overset{#4}{\underset{#1}{\overline{\nf}}} #3}

\title{Disjoint non-forking amalgamation in stable AECs}
\author{Jeremy Beard}

\begin{document}

\maketitle

\begin{abstract}
	The \emph{disjoint amalgamation property} (DAP), which asserts that all spans of a class of models can be amalgamated with minimal intersection, is an important property in the context of abstract elementary classes, with connections to both Grossberg's question \cite[Problem (5)]{sh576} and Shelah's categoricity conjecture \cite[Question 6.14(3)]{sh702}. We prove that, in a nice AEC $\K$ stable in $\lambda \geq \LS(\K)$ with a strong enough independence relation, all high cofinality $\lambda$-limit models are disjoint (non-forking) amalgamation bases.
	
	\begin{customthm}{\ref{disjoint-nf-amalgamation}}
		Let $\K$ be an AEC stable in $\lambda$, where $\K_\lambda$ has AP, JEP, and NMM, and let $\K'$ be some AC where $\Kkappalims \subseteq \K' \subseteq \K_\lambda$. Let $\dnf$ be an independence relation on $\K'$ satisfying uniqueness, existence, non-forking amalgamation, $\Kkappalims$-universal continuity* in $\K_\lambda$, and $(\geq \kappa)$-local character.
		
		Assume $M_0, M_1, M_2 \in \Kkappalims$, and that $M_0 \lek M_l$ and $a_l \in M_l$ for $l = 1, 2$. Then there exist $N \in \Kkappalims$ and $f_l : M_l \rightarrow N$ fixing $M_0$ for $l = 1, 2$ such that $\gtp(f_l(a_l)/f_{3-l}[M_{3-l}], N)$ $\dnf$-does not fork over $M_0$ and $f_1[M_1] \cap f_2[M_2] = M_0$. That is, $\dnf \upharpoonright \Kkappalims$ has disjoint non-forking amalgamation.
		
		In particular, every $M_0 \in \Kkappalims$ is a disjoint amalgamation base in $\K_\lambda$.
	\end{customthm}
	
	The hypotheses on the independence relation can be weakened (closer to $\lambda$-non-splitting in $\lambda$-stable AECs) if we are willing to give up the `non-forking' conditions of the amalgamation.
	
	This generalises \cite[Lemma 5.36]{vas19}, which applied only to categorical AECs, and more broadly continues the current effort to explore the structure of strictly stable abstract elementary classes.

\end{abstract}

\section{Introduction}

In the context of abstract elementary classes (AECs), the disjoint amalgamation property (DAP) is an important notion with connections to some of the central open problems in AECs, such as Grossberg's question \cite[Problem (5)]{sh576} (e.g. \cite[Lemma 2.4.10]{moab}) and Shelah's categoricity conjecture \cite[Question 6.14(3)]{sh702} (e.g. \cite[Fact 6.2]{vas19}). Given an AEC (or more generally an abstract class) $\K$, we say $M \in \K$ is a \emph{disjoint amalgamation base in $\K$} if for any two $\K$-extensions $M_1, M_2$ of $M$ where $M_1$ and $M_2$ have the same cardinality as $M$, there exist $\K$-embeddings $f_l : M_l \rightarrow N$ fixing $M$ such that $f_1[M_1] \cap f_2[M_2] = M$; that is, the intersection is the minimal possible one. We say a subclass $\K' \subseteq \K$ has \emph{disjoint amalgamation in $\K$} or \emph{DAP in $\K$} if every $M \in \K'$ is a disjoint amalgamation base in $\K$. In the case $\K = \K'$, we simply say $\K$ has \emph{disjoint amalgamation} or \emph{DAP}.

Given a first order complete theory $T$, the class of models of $T$ with elementary embeddings has DAP by an application of Robinson's Consistency Lemma \cite[Corollary 4.2.10]{moab}. However, in the more general setting of AECs, it is possible for DAP to fail, even assuming the amalgamation property (see Example \ref{example-DAP-stronger-than-AP}). In this paper, we show that if a nice AEC $\K$ which is stable in some $\lambda \geq \LS(\K)$ has a well behaved independence relation $\dnf$ on the high cofinality $\lambda$-limit models (see Definition \ref{universal-and-limit-definitions}), then not only are all high cofinality $\lambda$-limit models disjoint amalgamation bases in $\K$, but we can perform the disjoint amalgamation in a way that respects the non-forking relation on chosen singletons.

\begin{customthm}{\ref{disjoint-nf-amalgamation}}
	Let $\K$ be an AEC stable in $\lambda$, where $\K_\lambda$ has the amalgamation property, joint embeddding property, and no maximal models, and let $\K'$ be some AC where $\Kkappalims \subseteq \K' \subseteq \K_\lambda$. Let $\dnf$ be an independence relation on $\K'$ satisfying uniqueness, existence, non-forking amalgamation, $\Kkappalims$-universal continuity* in $\K_\lambda$, and $(\geq \kappa)$-local character.
	
	Assume $M_0, M_1, M_2 \in \Kkappalims$, and that $M_0 \lek M_l$ and $a_l \in M_l$ for $l = 1, 2$. Then there exist $N \in \Kkappalims$ and $f_l : M_l \rightarrow N$ fixing $M_0$ for $l = 1, 2$ such that $\gtp(f_l(a_l)/f_{3-l}[M_{3-l}], N)$ $\dnf$-does not fork over $M_0$ and $f_1[M_1] \cap f_2[M_2] = M_0$. That is, $\dnf \upharpoonright \Kkappalims$ satisfies disjoint non-forking amalgamation.
	
	In particular, every $M_0 \in \Kkappalims$ is a disjoint amalgamation base in $\K_\lambda$.
\end{customthm}

\begin{figure}[!ht]
\centering
\begin{circuitikz}
\tikzstyle{every node}=[font=\normalsize]
\node [font=\normalsize] at (5.25,12) {$a_1 \in$};
\node [font=\normalsize] at (8,9.2) {$a_2$};
\node [font=\normalsize, rotate=90] at (8,9.6) {$\in$};
\node [font=\normalsize] at (6,10) {$M$};
\draw [->, >=Stealth] (6,10.25) -- (6,11.75);
\draw [->, >=Stealth] (6.25,10) -- (7.75,10);
\node [font=\normalsize] at (8,10) {$M_2$};
\node [font=\normalsize] at (6,12) {$M_1$};
\node [font=\normalsize, color=blue] at (8.3,11) {$f_2$};
\node [font=\normalsize, color=blue] at (7,12.3) {$f_1$};
\draw [->, >=Stealth, dashed, color=blue] (6.25,12) -- (7.75,12);
\draw [->, >=Stealth, dashed, color=blue] (8,10.25) -- (8,11.75);
\node [font=\normalsize, color=blue] at (8,12) {$N$};
\node [font=\normalsize, color=blue] at (11,12) {$f_1(a_1) \dnf_{M}^N f_2[M_2]$};
\node [font=\normalsize, color=blue] at (11,10.8) {$f_2(a_2) \dnf_{M}^N f_1[M_1]$};
\node [font=\normalsize, color=blue] at (11,9.8) {$f_1[M_1] \cap f_2[M_2] = M$};
\end{circuitikz}

\caption{}
\label{disjoint-non-forking-amalgamation-diagram}
\end{figure}

In fact, the joint embedding property and no maximal models assumptions are unnecessary (see Remark \ref{remark-JEP-NMM-not-needed}).

The argument hinges on the theory of towers developed in \cite[\textsection 3]{bema}. Towers are essentially sequences of high cofinality limit models with named elements (see Definition \ref{tower-definitions}). An ordering $\lesst$ on towers can be defined (see \ref{tower-ordering-definintion}), which demands the singletons of towers have nice non-forking properties over the lesser towers. Of particular importance to us are \emph{reduced towers} - these are towers which have minimal intersections with the models $\lesst$-extending towers. 

We approach the proof of Theorem \ref{disjoint-nf-amalgamation} in stages, which we sketch here. In Subsection \ref{tower-amalgamation} we show that pairs of towers with the same bottom model can themselves be amalgamated - that is, you can form a matrix of models where the leftmost column and bottom row are isomorphic to two given towers, and the towers' singletons have nice non-forking properties with the matrix. In Subsection \ref{building-a-reduced-tower-between-m0-m1} we show how to build a reduced tower where the top model is a high cofinality limit model over the bottom model. We put these together in Subsection \ref{main-result-and-applications}: using that high cofinality limit models are isomorphic (Fact \ref{bemain}), we can ensure that (using the notation from the statement of Theorem \ref{disjoint-nf-amalgamation}) the bottom model of the tower is $M_0$, the top model is $M_1$, and $a_1$ is the bottom singleton of the tower. By applying tower amalgamation to this constructed tower and the tower $\langle M_0, M_2 \rangle ^\wedge \langle a_2 \rangle$, we get the desired non-forking amalgamation (the non-forking comes from the way we amalgamated the towers in the first step). Further, the `minimal intersection' condition of the reduced tower guarantees disjointness.

This approach is loosely inspired by \cite[Lemma 5.36]{vas19}, where this disjoint `non-forking' amalgamation is shown to hold for all models in an AEC which is categorical in $\lambda \geq \LS(\K)$ as well as being $\lambda$-superstable and $\lambda$-symmetric, and used a specific independence relation called $\lambda$-non-forking. Theorem \ref{disjoint-nf-amalgamation} has the advantages of applying to strictly stable AECs, as well as using a general independence notion rather than $\lambda$-non-forking. However, without the categoricity assumption, we can only prove that (high cofinality) $\lambda$-limit models are disjoint (non-forking) amalgamation bases, rather than all models of size $\lambda$.

In Subsection \ref{disjoint-amalgamation-from-a-weaker-relation}, we show that in a more general setting (where our independence relation satisfies only `weak' versions of uniqueness, extension, and non-forking amalgamation, similar to $\lambda$-non-splitting in $\lambda$-stable AECs), all high cofinality $\lambda$-limit models are disjoint amalgamation bases. However, we have not been able to show the non-forking of the amalgamation can be guaranteed. This is because, in this context, we must use the notion of towers from \cite{beard3} (in this paper called \emph{weak towers}). These interact more fruitfully with the weaker forms of uniqueness, extension, and non-forking amalgamation our independence relation assumes, but are trickier to work with. It remains open whether an analogous disjoint \emph{non-forking} amalgamation can be performed in this setting (see Question \ref{weak-question}).

This paper is split into three sections. Section 2 provides the necessary background. In particular, the theory of towers is introduced - we only state the relevant definitions and results, but an interested reader can consult \cite[\textsection 3]{bema} for a more thorough introduction. Section 3 contains all new arguments - in particular, Subsections 3.1, 3.2, and 3.3 together prove Theorem \ref{disjoint-nf-amalgamation}. Subsection 3.4 introduces the theory of towers from \cite{beard3} which we call \emph{weak towers} to distinguish from those used in the rest of the paper, before proving the disjoint amalgamation result in this weaker setting.

This paper was written while the author was working on a Ph.D. thesis under the direction of Rami Grossberg at Carnegie Mellon University, and the author would like to thank Professor Grossberg for his guidance and assistance in his research in general and in this work specifically.

\section{Preliminaries}

We will extensively use material from \cite[\textsection 3]{bema}, particularly the theory of reduced and full towers, and the main result of that section (\cite[Theorem 3.1]{bema}). We present the relevant material in this section for the convenience of the reader. A reader familiar with \cite[\textsection 3]{bema} can skip this section.

For now, we assume some background knowledge of abstract elementary classes (AECs) - see for example \cite{baldwinbook}.

\subsection{Basic notions, notations, and conventions}

An \emph{abstract class} (AC) is a pair $\K = (K, \lek)$ class of models $K$ closed under isomorphisms, with a partial order $\lek$ which strengthens the substructure relation $\subseteq$ and respects isomorphisms. Often $\K$ will be an AEC. Given a model $M$, $|M|$ denotes its universe, and $\|M\|$ its cardinality. We use the standard abuses of notation $M \in \K$ instead of $M \in K$, and sometimes use $M$ in place of the universe $|M|$, e.g. $a \in M$. $M \cong N$ signals that the models $M$ and $N$ are isomorphic, and $M \underset{M_0}{\cong} N$ denotes that there is an isomorphism from $M$ to $N$ fixing the common $\lek$-substructure $M_0$. We use $\alpha, \beta, \gamma \dots$ to denote ordinals, often reserving $\delta$ for limits, $\lambda$ for cardinals, and $\kappa$ for a regular cardinals. Given two sequences $\sigma$ and $\tau$, their concatenation is $\sigma ^\wedge \tau$. Given an AC $\K = (K, \lek)$, we say $\K' = (K', \leq_{\K'})$ is a sub-AC of $\K$ if $\K'$ is an AC, $K' \subseteq K$, and $\leq_{\K'} = \lek \upharpoonright (K')^2$. When $\K$ is an AC and $\lambda$ is a cardinal, $\K_\lambda$ is the sub-AC of $\K$ with underlying class $K_\lambda = \{M \in K : \|M\| = \lambda\}$ ordered by the restriction of $\lek$ to $K_\lambda$. $K_{\geq \lambda}$ and $\K_{\geq \lambda}$ are defined similarly. Given $M \lek N$ and $a \in N$, $\gtp(a/M, N)$ denotes the Galois type of $a$ over $M$ in $N$. The set of (1-ary) Galois types over $M$ is denoted $\gS(M)$. We say $\K$ is $\lambda$-stable if for all $M \in \K$ with $\|M\| = \lambda$, $|\gS(M)| \leq \lambda$.

An AC $\K$ has the \emph{amalgamation property} (AP) if for all $M_0, M_1, M_2$ where $M_0 \lek M_l$ for $l = 1, 2$, there exist $N$ and $\K$-embeddings $f_l : M_l \rightarrow N$ fixing $M_0$ for $l = 1, 2$. $\K$ has the \emph{joint embedding property} (JEP) if every two models $M_1$ and $M_2$ have $\K$-embeddings into a common structure $N$. $\K$ has \emph{no maximal models} (NMM) if for all $M \in \K$, there exists $N \in \K$ where $M <_\K N$.

\begin{definition}\label{disjoint-ap-base}
    Let $\K$ be an AC. We say that $M \in \K$ is a \emph{disjoint amalgamation base in $\K$} if for every $M_1, M_2 \in \K_{\|M\|}$ such that $M \lek M_l$ for $l = 1, 2$, there exist $N$ and $f_l:M_l \rightarrow N$ fixing $M$ for $l = 1, 2$ such that $f_1[M_1] \cap f_2[M_2] = M$.
    
    If $\K$ is clear from context, we may omit it. If every $M \in \K$ is a disjoint amalgamation base, we say $\K$ satisfies the \emph{disjoint amalgamation property} or \emph{DAP}.
\end{definition}

Compare this to the definition of the amalgamation property - disjoint amalgamation requires the extra `minimal intersection' condition, but note disjoint amalgamation we only ask that our conclusion holds for models of the same cardinality. Other authors may not enforce the cardinality condition, but since in this paper we will only discuss AP and DAP holding in $\K_\lambda$ for an AC $\K$ and some cardinal $\lambda$, the notions will be the same for our main results. In light of this, DAP can be viewed as a stronger condition than AP.

In first order model theory, AP and DAP are actually equivalent by an application of Robinson's consistency lemma \cite[Corollary 4.2.10]{moab}. However, in the AEC setting DAP is strictly stronger than AP, since there exists a very nice AEC satisfying AP (and JEP, NMM, and stability in all cardinals) which does not satisfy the DAP in any cardinal (see Example \ref{example-DAP-stronger-than-AP}).

DAP has a couple of alternative formulations.

\begin{lemma}\label{disjoint-ap-non-fixing-def}
	Let $\K$ be an AC, and $M \in \K$. Then $M$ is a disjoint amalgamation base if and only if for every $M_1, M_2 \in \K_{\|M\|}$ such that $M \lek M_l$ for $l = 1, 2$, there exist $N$ and $f_l:M_l \rightarrow N$ for $l = 1, 2$ such that $f_1[M_1] \cap f_2[M_2] = f[M]$.
\end{lemma}

\begin{proof}
	It is clear if $M$ is a disjoint amalgamation base, then any $f_1, f_2$ witnessing this are sufficient. 
	
	For the other direction, given $M_1, M_2 \in \K_{\|M\|}$ such that $M \lek M_l$ for $l = 1, 2$, let $N$ and $f_l:M_l \rightarrow N$ for $l = 1, 2$ be such that $f_1[M_1] \cap f_2[M_2] = f[M]$. Take $f_1' : M_1' \rightarrow N$ an isomorphism extending $f_1$. Then taking $g_1 : M_1 \rightarrow M_1'$ to be the identity, and $g_2 = g_1^{-1} \circ f_2: M_2 \rightarrow M_1'$, we have that $g_1, g_2$ fix $M$ and $g_1[M_1] \cap g_2[M_2] = g_1[M] = M$ as desired.
\end{proof}

The following reformulation essentially says that, since the places where a disjoint amalgamation sends the elements of $M_1 \setminus M_0$ and $M_2 \setminus M_0$ don't overlap at all, if $M_1\cap M_2 = M_0$ already, we may assume $M_1 \setminus M_0$ and $M_2 \setminus M_0$ are fixed.

\begin{lemma}\label{disjoint-ap-with-no-f}
	Suppose $\K$ is an AC, and $M_0 \in \K$. Then $M_0$ is a disjoint amalgamation base if and only if whenever $M_1, M_2 \in \K_{\|M_0\|}$ such that $M_0 \lek M_l$ for $l = 1, 2$ and $M_1 \cap M_2 = M_0$, then there exists $N \in \K$ such that $M_1, M_2 \lek N$.
\end{lemma}

\begin{proof}
	The reverse implication is trivial, so suppose $M_0$ is an amalgamation base. Then there exists $N^* \in \K$ with $N^* \geq M_2$ and $f:M_1 \rightarrow N^*$ fixing $M_0$ such that $f[M_1] \cap M_2 = M_0$. Relabelling $|N^*|$ if necessary, without loss of generality, we may assume that $(f[|M_1|] \setminus |M_0|) \cap |M_1| = \varnothing$ (move $(f[|M_1|] \setminus |M_0|)$ to some other set disjoint from both $M_1, N^*$, and replace $N^*$ by the induced structure, similar to below). Define a bijective function $g : |N^*| \rightarrow A$ where
	\begin{enumerate}
		\item $g(a) = f^{-1}(a)$ if $a \in f[|M_1|]$
		\item $g(a) = a$ if $a \in |N^*| \setminus f[|M_1|]$.
	\end{enumerate}
	Note that if $g$ fixes $M_0$ as $f$ fixes $M_0$, so also $g$ fixes $M_2$ as $M_2 \cap f[M_1] = M_0$. Formally, $A = |M_1| \cup (|N^*| \setminus f[|M_1|])$. The bijection $g$ defines a structure $N$ with $|N| = A$ where $g : N^* \cong N$. By the isomorphism axioms of abstract classes, note that $N \in \K$, $g[f[M_1]] = M_1 \lek N$, and $g[M_2] = M_2 \lek N$. So $N$ is as desired.
\end{proof}

The following example shows that, in the setting of AECs, DAP is a strictly stronger property than AP.

\begin{example}\label{example-DAP-stronger-than-AP}
	Let $\K = (K, \lek)$ be the AC determined as follows:
	\begin{enumerate}
		\item $K$ is the class of all models $M = (|M|, E)$ where $E$ is an equivalence relation on $|M|$ such that each equivalence class of $E$ has at most two elements
		\item $\lek$ is the submodel relation $\subseteq$
	\end{enumerate}
	It is not hard to see that this is an AEC with $\LS(\K) = \aleph_0$, stable in all $\mu \geq \aleph_0$, and satisfies AP, JEP, and NMM. However, the disjoint amalgamation property fails in not only $\K$, but in $\K_\mu$ for every cardinal $\mu \geq \aleph_0$.
	
	To see this, take $M = (\mu \setminus \{1, 2\}, E)$ where $E_0$ has equivalence classes $\{i\}$ for $i \in |M|$ (all classes are singletons). For $l = 1, 2$, take $M_l = (|M| \cup \{l\}, E_l)$ where $E_l$ has equivalence classes $\{0, l\}$ and $\{i\}$ for $i \in [3, \mu)$ (that is, we added a new element $l$ to the class of $0$). It is clear that $M, M_1, M_2 \in \K$ and $M \lek M_l$ for $l = 1, 2$. Supposing for contradiction $M$ were a disjoint amalgamation base, by Lemma \ref{disjoint-ap-with-no-f}, there would be a model $N = (|N|, E_N) \in \K$ with $M_l \lek N$ for $l = 1, 2$. But then $0 E_N 1$ and $0 E_N 2$, meaning the equivalence class of $0$ in $N$ had at least three elements, contradicting the definition of $\K$.
\end{example}

\subsection{Limit models}

\begin{definition}\label{universal-and-limit-definitions}
    Let $\K$ be an AC, and $\lambda \geq \LS(\K)$ such that $\K_\lambda$ has AP, and $\delta < \lambda^+$ a limit ordinal.
    \begin{enumerate}
        \item Given $M \lek N$ with $\|M\| = \|N\|$, we say $N$ is \emph{universal over $M$} if for every $N' \in \K$ with $M \lek N'$ and $\|N'\| = \|M\|$, there exists $f:N' \rightarrow N$ fixing $M$. We write $M \lek^u N$ in this case.
        \item Given $M \lek N$, we say $N$ is a \emph{$(\lambda, \delta)$-limit model over $M$} if there exists a $\lek^u$-increasing sequence $\langle M_i : i \leq \delta \rangle$ in $\K_\lambda$ such that $M_0 = M$ and $N = M_\delta$. We will sometimes refer to $\delta$ as the \emph{length} of the limit model.
        \item Given $\kappa < \lambda^+$ regular, we say $N$ is a \emph{$(\lambda, \geq \kappa)$-limit model over $M$} if $N$ is a $(\lambda, \delta)$-limit model over $M$ for some $\delta < \lambda^+$ with $\cof(\delta) \geq \kappa$. We write $M \underhlim N$ in this case.
        \item We say $N$ is a \emph{$(\lambda, \delta)$-limit model} if it is a $(\lambda, \delta)$-limit model over some $M$. Similarly define $(\lambda, \geq \kappa)$-limit model, and $\lambda$-limit model (over $M$). We will sometimes omit $\lambda$ also when it is clear from context.
        \item $\kkappalims$ is the class of $(\lambda, \geq \kappa)$-limit models in $\K$. $\Kkappalims$ is the AC given by the restriction of $\lek$ to $\kkappalims$.
    \end{enumerate}
\end{definition}

Note that if $N$ is a limit model over $M$, $N$ is universal over $M$ also by AP. So $\underhlim \subseteq \lek^u \subseteq \lek$.

Limit models of any limit length $\delta < \lambda^+$ exist in any $\lambda$-stable AEC.

\begin{fact}[{\cite[2.9, 2.12]{grva06a}}]\label{universal-and-limit-extensions-exist}
    Let $\K$ be an AEC and $\lambda \geq \LS(\K)$. Suppose $\K_\lambda$ has AP. If $\langle M_i : i \leq \lambda \rangle$ is a $\lek$-increasing sequence in $\K_\lambda$, and for all $i < \lambda$ $M_{i+1}$ realises all types over $M_i$, then $M_\lambda$ is universal over $M_0$.
    
    If $\K$ is also stable in $\lambda \geq \LS(\K)$ with AP in $\K_\lambda$, then any model $M \in \K_\lambda$ has a universal extension $N \in \K_\lambda$. Moreover, if $\delta < \lambda^+$ is limit, then for any $M \in \K_\lambda$ there exists a $(\lambda, \delta)$-limit model $N$ over $M$.
\end{fact}

It is well known that limit models (over the same base) of the same cofinality of length are isomorphic (over the base).

\begin{fact}[{\cite[1.3.6]{shvi99}}]\label{cofinality-iso}
    Let $\K$ be an AC where $\K_\lambda$ has AP. Let $M, N_1, N_2 \in \K_\lambda$ where $N_l$ is a $(\lambda, \delta_l)$-limit model over $M$ for $l = 1, 2$. Suppose $\cof(\delta_1) = \cof(\delta_2)$. Then there is an isomorphism $f:N_1 \cong N_2$ fixing $M$.

    Furthermore, if $\K_\lambda$ has JEP, then given any $N_1, N_2 \in \K_\lambda$ where $N_l$ is a $(\lambda, \delta_l)$-limit model for $l = 1, 2$ (not necessarily over the same base), then there is an isomorphism $f:N_1 \cong N_2$.
\end{fact}

The question of when limit models are isomorphic is an important one, and more can be said under stronger assumptions (see Fact \ref{bemain}), but we save this for after our discussion of towers.

\subsection{Independence relations}

Independence relations are relations that mimic the properties of non-forking in first order stable theories. Such relations do not exist in all AECs, but can be extremely useful when they do exist, as is the case for our main result. Throughout this paper, the term will refer to relations on singletons with a triple of models. We follow the notations and conventions of \cite{bema} and \cite{beard3}. See \cite[\textsection 2.3]{beard3} for a little history.

\begin{definition}\label{independence-relation-definitions}
    Let $\K$ be an AC. Let $\dnf$ be a relation on tuples $(M_0, a, M, N)$ where $M_0 \lek M \lek N$ and $a \in N$. We write $a \dnf_{M_0}^N M$ as a shorthand for $(M_0, a, M, N) \in \dnf$.

    We say $\dnf$ is an \emph{independence relation} when it satisfies
    \begin{enumerate}
        \item \emph{Invariance:} for all $M_0 \lek M \lek N$, $a \in N$, and $f : N \cong N'$, if $a \dnf_{M_0}^N M$ then $f(a) \dnf_{f[M_0]}^{N'} f[M]$ (that is, $\dnf$ respects isomorphisms)
        \item \emph{Monotonicity:} for all $M_0 \lek M_1 \lek M \lek N_1 \lek N \lek N_2$ and $a \in N_0$, if $a \dnf_{M_0}^N M$, then $a \dnf_{M_0}^{N_2} M_1$ and $a \dnf_{M_0}^{N_1} M_1$ (that is, we can let $M$ shrink and let $N$ either grow or shrink, and the relation still holds)
        \item \emph{Base monotonicity:} for all $M_0 \lek M_1 \lek M \lek N$ and $a \in N$, if $a \dnf_{M_0}^N M$, then $a \dnf_{M_1}^N M$ (that is, we can grow $M_0$ and the relation still holds)
    \end{enumerate}
\end{definition}

We can define a notion of non-forking on types from any independence relation.

\begin{definition}
    Let $\K$ be an AC. Let $\dnf$ be an independence relation on $\K$, $M_0 \lek M$, and $p \in \gS(M)$. We say $p$ $\dnf$-does not fork over $M_0$ if there exist $N \in \K$ and $a \in N$ such that $p = \gtp(a/M, N)$ and $a \dnf_{M_0}^N M$.
\end{definition}

Note that by invariance and monotonicity, this is equivalent if we replace the existence of such $a$ and $N$ with `for all' such $a$ and $N$.

Independence relations may (or may not) satisfy some collection of the properties below, familiar from non-forking in first order stable theories.

\begin{definition}\label{dnf-properties}
    Let $\K$ be an AC, $\kappa$ a regular cardinal, $\delta$ a limit ordinal, and $\dnf$ an independence relation on $\K$. We say $\dnf$ satisfies:
    \begin{enumerate}
        \item \emph{Uniqueness} if whenever $M \lek N$ and $p_1, p_2 \in \gS(N)$ such that $p_l$ $\dnf$-does not fork over $M$ for $l = 1, 2$ and $p_1 \upharpoonright M = p_2 \upharpoonright M$, we have $p_1 = p_2$.
        \item \emph{Weak uniqueness} if whenever $M_0 \lek^u M \lek N$ and $p_1, p_2 \in \gS(N)$ such that $p_l$ $\dnf$-does not fork over $M_0$ for $l = 1, 2$ and $p_1 \upharpoonright M = p_2 \upharpoonright M$, we have $p_1 = p_2$.
        \item \emph{Extension} if whenever $M_0 \lek M \lek N$ and $p \in \gS(M)$ where $p$ $\dnf$-does not fork over $M_0$, there exists $q \in \gS(N)$ such that $q$ $\dnf$-does not fork over $M_0$ and $q \upharpoonright M = p$.
        \item \emph{Weak extension} if whenever $M_0 \lek^u M \lek N$ and $p \in \gS(M)$ where $p$ $\dnf$-does not fork over $M_0$, there exists $q \in \gS(N)$ such that $q$ $\dnf$-does not fork over $M_0$ and $q \upharpoonright M = p$.
        \item \emph{Transitivity} if whenever $M_0 \lek M \lek N$ and $p \in \gS(N)$ such that $p \upharpoonright M$ $\dnf$-does not fork over $M_0$ and $p$ $\dnf$-does not fork over $M$, then $p$ $\dnf$-does not fork over $M_0$.
        \item \emph{Existence} if whenever $M \in \K$ and $p \in \gS(M)$, $p$ $\dnf$-does not fork over $M$.
        \item \emph{$\delta$-local character} if whenever $\langle M_i : i \leq \delta \rangle$ is a $\lek^u$-increasing sequence in $\K$ continuous at $\delta$ and $p \in \gS(M_\delta)$, then there exists $i < \delta$ such that $p$ $\dnf$-does not fork over $M_i$.
        \item \emph{$(\geq \kappa)$-local character} if $\dnf$ has $\delta$-local character for all $\delta$ where $\cof(\delta) \geq \kappa$.
        \item \emph{$\delta$-universal continuity} if whenever $\langle M_i : i \leq \delta \rangle$ is a $\lek^u$-increasing sequence in $\K$ continuous at $\delta$ and $p \in \gS(M_\delta)$ such that $p \upharpoonright M_i$ $\dnf$-does not fork over $M_0$ for all $i < \delta$, then $p$ $\dnf$-does not fork over $M_0$.
        \item \emph{$(\geq \kappa)$-universal continuity} if $\dnf$ has $\delta$-universal continuity for all limits $\delta$ where $\cof(\mu) \geq \kappa$.
        \item \emph{Universal continuity} if $\dnf$ has $(\geq \aleph_0)$-universal continuity.
        \item \emph{non-forking amalgamation} if for any $M_0, M_1, M_2 \in \K$ where $M_0 \lek M_l$ and $a_l \in M_l$ for $l = 1, 2$, there exist $N \in \K$ and $f_l : M_l \rightarrow N$ fixing $M_0$ for $l = 1, 2$ such that $\gtp(f_l(a_l)/f_{3-l}[M_{3-l}], N)$ $\dnf$-does not fork over $M_0$ for $l = 1, 2$ (that is, we can amalgamate $M_1$ and $M_2$ over $M_0$ such that the types of (the images of) the singletons do not fork over (the images of) the `opposite' models).
        \item\label{disjoint-nf-amalgamation-def} \emph{disjoint non-forking amalgamation} if for any $M_0, M_1, M_2 \in \K$ where $M_0 \lek M_l$ and $a_l \in M_l$ for $l = 1, 2$, there exist $N \in \K$ and $f_l : M_l \rightarrow N$ fixing $M_0$ for $l = 1, 2$ such that $\gtp(f_l(a_l)/f_{3-l}[M_{3-l}], N)$ $\dnf$-does not fork over $M_0$ for $l = 1, 2$, and $f_1[M_1] \cap f_2[M_2] = M_0$ (that is, we can perform non-forking amalgamation with the minimal intersection).
    \end{enumerate}
\end{definition}

\begin{remark}\label{existence-transitivity-remark}
    Many of the above properties are related - for example, it is well known that uniqueness and extension imply transitivity (see e.g. \cite[Corollary III.4.4]{sh:c} or \cite[2.16]{bemav1}), and if $\K = \K'_{(\lambda, \geq \kappa)}$ for some AEC $\K'$ stable in $\lambda \geq \LS(\K')$, then $(\geq \kappa)$-local character implies existence \cite[3.10]{bema}.

\end{remark}

\begin{remark}
	By the same method as Lemma \ref{disjoint-ap-with-no-f}, disjoint non-forking amalgamation is equivalent if we enforce that $f_l = \operatorname{id}_{M_l}$ for $l = 1, 2$. That said we will generally work with the version in Definition \ref{independence-relation-definitions}(\ref{disjoint-nf-amalgamation-def}) since this allows us more flexibility when proving the property holds.
\end{remark}

When $\K$ is a nice sub-AC of an AEC but does not contain arbitrary unions, the usual notion of universal continuity (or $(\geq \aleph_0)$-continuity in the notation above) is less useful than the `universal continuity*' defined originally in \cite[3.2]{bema} (which is equivalent when arbitrary unions of universal chains \emph{are} in the AC). While writing \cite{beard3} the author realised there is a minor missing assumption in \cite[3.2]{bema}, which is that the types computed in the larger class are compatible with those computed in the smaller class. We formalise this with the following definition.

\begin{definition}[{\cite[Definition 2.11]{beard3}}]
    Let $\K$ be an AC with sub-AC $\K'$, both with AP. We say \emph{$\K'$ respects types from $\K$} if 
    \begin{enumerate}
        \item for all $M \in \K'$, $\Phi_M:\gS_{\K'}(M) \rightarrow \gS_\K(M)$ given by $\Phi_M(\gtp_{\K'}(a/M, N)) = \gtp_\K(a/M, N)$ for $M \leq_{\K'} N$ and $a \in N$ is a well defined bijection
        \item when $M \leq_{\K'} N$, $p \in \gS_{\K'}(M)$ and $q \in \gS_{\K'}(N)$, then $p \subseteq q$ if and only if $\Phi_M(p) \subseteq \Phi_N(q)$.
    \end{enumerate}
\end{definition}

The upshot of classes respecting types is that we can more or less forget about which class the types over models in $\K'$ are computed in. As such we will not normally distinguish between $p$ and $\Phi_M(p)$. In particular, if $\dnf$ is an independence relation on $\K'$, we may ask whether $p$ $\dnf$-does not fork over $M_0$ when $p \in \gS_{\K}(M)$ and $M_0 \leq_{\K'} M$.

\begin{definition}[{\cite[Definition 3.2]{bema}}]\label{continuity-star-def}
    Let $\K$ be an AC with sub-AC $\K'$, both with AP, and $\K'$ respects types from $\K$. Let $\dnf$ be an independence relation on $\K'$. We say that $\dnf$ has \emph{universal continuity* in $\K$} if whenever 
    \begin{enumerate}
    	\item $\delta$ is a limit ordinal
    	\item $\langle M_i : i < \delta \rangle$ is a $\lek^u$-increasing sequence in $\K'$ where $\bigcup_{i<\delta} M_i \in \K$
    	\item $M \in \K'$ is such that $M_i \lek M$ for all $i < \delta$
    	\item $\langle p_i : i < \delta \rangle$ is an increasing sequence of types with $p_i \in \gS(M_i)$ such that $p_i$ $\dnf$-does not fork over $M_i$ for all $i<\delta$
    \end{enumerate} 
    
    then there is a unique $p \in \gS(\bigcup_{i<\delta} M_i)$ such that $p_i \subseteq p$ for all $i < \delta$.
\end{definition}

Formally, in the above, $p_i \in \gS_{\K'}(M)$, $p \in \gS_\K(\bigcup_{i < \kappa} M_i)$, and the final conclusion is actually $\Phi_{M_i}(p_i) \subseteq p$ for all $i < \delta$, but perhaps that obscures the spirit of the definition. 

\begin{remark}
    Note that all cases where universal continuity* is used in \cite{bema}, $\Kkappalims \subseteq \K' \subseteq \K_\lambda$ for some AEC $\K$ stable in $\lambda \geq \LS(\K)$ with AP, JEP, and NMM in $\K_\lambda$. Since for every $M \in \K_\lambda$ there is $N \in \Kkappalims$ with $M \lek N$, this implies $\K'$ respects types from $\K$, so this slightly revised definition corrects \cite[3.2]{bema} without causing further issues with those results. See \cite[Remark 2.16]{beard3} for more details.
\end{remark}

For the main body of this paper, we only concern ourselves with the case when $\K$ is an AEC stable in $\lambda$, and $\K'$ is an AC with $\Kkappalims \subseteq \K' \subseteq \K_\lambda$.

\begin{fact}[{\cite[Lemma 2.12]{beard3}}]
	Suppose $\K$ is an AEC stable in $\lambda \geq \LS(\K)$, and $\K_\lambda$ has AP. Suppose $\K'$ is an AC where $\Kkappalims \subseteq \K' \subseteq \K_\lambda$. Then
	\begin{enumerate}
		\item Embeddings in $\K'$ are just $\K$-embeddings between models in $\K'$
		\item $\K'$ respects types from $\K$
	\end{enumerate}
\end{fact}

\begin{definition}
	Let $\dnf$ be an independence relation on an AC $\K$ with AP. Suppose $\K'$ is an AC that respects types from $\K$. The \emph{restriction of $\dnf$ to $\K'$} is the relation $\dnf \upharpoonright \K'$ where $a \overset{N}{\underset{M_0}{(\dnf \upharpoonright\K')}} M$ if and only if $a \dnf_{M_0}^N M$ and $M_0, M, N \in \K'$.
\end{definition}

\begin{remark}
	Suppose $\K'$ be an AC that respects types from another AC $\K$, and $\dnf$ is an independence relation on $\K$. Then $\dnf \upharpoonright \K'$ is an independence relation on $\K'$. Since types in $\K'$ are practically just types from $\K$ over models in $\K'$, the  can effectively be viewed as a relation on the types $p \in \gS_\K(M)$ for $M \in \K'$ over models in $\K'$.
\end{remark}

\begin{definition}
    Let $\K$ be an AEC stable in $\lambda \geq \LS(\K)$ where $\K_\lambda$ has AP, and $\K'$ an AC with AP such that $\Kkappalims \subseteq \K' \subseteq \K_\lambda$. Let $\dnf$ be an independence relation on $\K'$. Then we say $\dnf$ has \emph{$\Kkappalims$-universal continuity* in $\K$} if the restriction of $\dnf$ to $\Kkappalims$ has universal continuity* in $\K$.
    
    We will sometimes omit `in $\K$' when $\K$ is obvious.
\end{definition}

In applications, typically $\K' = \K$ or $\K' = \Kkappalims$. Note that in the latter case, $(\geq \aleph_0)$-universal continuity is practically only as good as $(\geq \kappa)$-universal continuity (as the unions of shorter chains may not be in $\Kkappalims$). However $\Kkappalims$-universal continuity* implies not only $(\geq \kappa)$-universal continuity, but also gives us something like continuity for shorter chains, provided we are allowed to peek back inside $\K_\lambda$.

\begin{fact}[{\cite[3.11]{bema}}]\label{continuity-star-implies-high-continuity}
    Let $\K$ be an AEC stable in $\lambda \geq \LS(\K)$, and $\K'$ an AC with $\Kkappalims \subseteq \K' \subseteq \K_\lambda$, where $\K_\lambda$ has AP. Let $\dnf$ be an independence relation on $\K'$ with uniqueness, existence, and $\Kkappalims$-universal continuity*. Then $\dnf$ has $(\geq \kappa)$-universal continuity.
\end{fact}

In fact, when $\dnf$ is defined on all models (or just all limit models) in the above, we have that universal continuity is equivalent to universal continuity* in $\K_\lambda$.

\begin{fact}[{\cite[3.12]{bema}}]\label{continuity-implies-continuity-star}
    Let $\K$ be an AEC stable in $\lambda \geq \LS(\K)$, and $\K'$ an AC with $\K_{(\lambda, \geq \aleph_0)} \subseteq \K' \subseteq \K_\lambda$, where $\K_\lambda$ has AP. Let $\dnf$ be an independence relation on $\K'$ with uniqueness, existence, and $\Kkappalims$-universal continuity*. Then $\dnf$ has universal continuity if and only if $\dnf$ has universal continuity* in $\K_\lambda$.
\end{fact}

Now we have the terminology to give the underlying hypotheses for the main theorem. 

\begin{hypothesis}\label{general-hypothesis}
    Let $\K$ be an AEC stable in $\lambda$, where $\K_\lambda$ has AP, JEP, and NMM, and let $\K'$ be some AC where $\Kkappalims \subseteq \K' \subseteq \K_\lambda$. Let $\dnf$ be an independence relation on $\K'$ satisfying uniqueness, existence, non-forking amalgamation, $\Kkappalims$-universal continuity* in $\K_\lambda$, and $(\geq \kappa)$-local character.
\end{hypothesis}

We state it in this generality to make it easy to apply, but in fact we will only ever use the restriction of $\dnf$ to $\Kkappalims$. Because of this, we use the following hypothesis throughout the proof.

\begin{hypothesis}\label{proof-hypothesis}
    Let $\K$ be an AEC stable in $\lambda$, where $\K_\lambda$ has AP, JEP, and NMM. Let $\dnf$ be an independence relation on $\Kkappalims$ satisfying uniqueness, existence, non-forking amalgamation, universal continuity* in $\K_\lambda$, and $(\geq \kappa)$-local character.
\end{hypothesis}

\begin{remark}\label{general-hypothesis-implies-proof}
    Hypothesis \ref{proof-hypothesis} is identical to \cite[3.7]{bema}. Also, Hypothesis \ref{general-hypothesis} implies that Hypothesis \ref{proof-hypothesis} holds for the restriction of $\dnf$ to $\Kkappalims$.
\end{remark}

\begin{remark}
    Note that under Hypothesis \ref{proof-hypothesis}, $\dnf$ has existence, transitivity, and $(\geq \kappa)$-universal continuity by Remark \ref{existence-transitivity-remark} and Fact \ref{continuity-star-implies-high-continuity}.
\end{remark}

\subsection{Towers}\label{subsection-towers}

Throughout this subsection, assume Hypothesis \ref{proof-hypothesis} holds for some $\K$ stable in $\lambda \geq \LS(\K)$, $\kappa < \lambda^+$ regular, and $\dnf$ an independence relation on $\Kkappalims$.

\begin{notation}
    Let $I$ be a well ordered set. Define $I^-$ as follows:
    \begin{enumerate}
        \item if $\operatorname{otp}(I)$ is $0$ or limit, $I^- = I$
        \item if $\operatorname{otp}(I)$ is successor and $i_1$ is the final element of $I$, $I^- = I \setminus \{i_1\}$ (with the restriction of the ordering of $I$).
    \end{enumerate}
    That is, $I^-$ is $I$ with the final element removed, if there is one. It is the largest set of $i \in I$ such that $i+1$ exists in $I$. Note that is $\alpha$ is an ordinal, then $\alpha^-$ is as well.
\end{notation}

\begin{definition}\label{tower-definitions}
    A \emph{tower} is a sequence $\langle M_i : i \in I \rangle ^ \wedge \langle a_i : i \in I^- \rangle$ such that
    \begin{enumerate}
        \item $I$ is a well ordered set with $\operatorname{otp}(I) < \lambda^+$
        \item $\langle M_i : i \in I \rangle$ is a $\lek$-increasing sequence of $(\lambda, \geq \kappa)$-limit models
        \item for all $i \in I^-$, $a_i \in M_i$.
    \end{enumerate}
    We call $I$ the \emph{index} of $\calt$.

\end{definition}

\begin{remark}\label{towers-exist}
    Although in \cite[3.17]{bema} we ask $a_i \in |M_{i+1}| \setminus |M_i|$, we lift this restriction here as it is not actually used in the proof of \cite[3.1]{bema} or the theory of towers developed to prove it (the restriction in that paper was inherited from the tower presentation in \cite{vas19}). By removing this restriction, we no longer need to assume NMM (this was only needed to enforce $a_i \notin M_i$ at various points before to match the tower definition) - but we can remove this anyway using Remark \ref{remark-JEP-NMM-not-needed}.
\end{remark}

\begin{definition}

    Given a tower $\calt = \langle M_i : i \in I \rangle ^ \wedge \langle a_i : i \in I^- \rangle$, we say

    \begin{enumerate}
        \item $\calt$ is \emph{continuous at $i \in I$} if $\langle M_i : i \in I \rangle$ is continuous at $i \in I$ (that is, $M_i = \bigcup_{r < i} M_r$)
        \item given an order $\leq^*$ on $\kkappalims$ where $\leq^* \subseteq \lek$, $\calt$ is \emph{$\leq^*$-increasing} if and only if $\langle M_i : i \in I\rangle$ is $\leq^*$-increasing
        \item given an order $\leq^*$ on $\kkappalims$ where $\leq^* \subseteq \lek$, $\calt$ is \emph{strongly $\leq^*$-increasing} if and only if $\langle M_i : i \in I \rangle$ is (that is, for all $i \in I$, $\bigcup_{r<i} M_r \leq^* M_i$)
        \item given $I_0 \subseteq I$, the \emph{restriction of $\calt$ to $I_0$} is $\calt \upharpoonright I_0 = \langle M_i : i \in I_0 \rangle ^\wedge \langle a_i : i \in (I_0)^-\rangle$
    \end{enumerate}
\end{definition}

Given a model $M_0 \in \Kkappalims$, for any $\alpha < \lambda^+$ there is a strongly $\underhlim$-increasing sequence $\langle M_i : i < \lambda \rangle$ and a sequence $\langle a_i : i < \alpha \rangle$ such that $\calt = \langle M_i : i < \alpha \rangle ^\wedge \langle a_i : i < \alpha^-\rangle$ is a tower by repeatedly taking $\underhlim$-extensions over the union of the previous models and taking $a_i$ arbitrarily. Or, you can modify the construction to make it only $\underhlim$-increasing, but continuous at all $i < \alpha$ with $\cof(i) \geq \kappa$, by taking unions at these steps in the construction.

\begin{definition}\label{tower-ordering-definintion}
    Let $\calt = \langle M_i : i \in I \rangle ^ \wedge \langle a_i : i \in I^- \rangle$ and $\calt = \langle M_i' : i \in I' \rangle ^ \wedge \langle a_i' : i \in (I')^- \rangle$ be two towers. We define the \emph{tower ordering} $\lesst$ by $\calt \lesst \calt'$ if and only if
    
    \begin{enumerate}
        \item $I \subseteq I'$
        \item for all $i \in I$, $M_i \lek^u M_i'$
        \item for all $i \in I^-$, $a_i = a_i'$
        \item for all $i \in I^-$, $\gtp(a_i/M_i', M_{i+_I1}')$ $\dnf$-does not fork over $M_i$.
    \end{enumerate}

    We also define a stronger tower ordering $\lesstrong$ as above, but replacing clause (2) with 
    \begin{enumerate}
        \item[2*.] for all $i \in I$, $M_i \underhlim M_i'$.
    \end{enumerate}
\end{definition}

Note that because of condition (3), when towers are comparable by the ordering, we don't need to use different symbols for the singletons of different towers. Also note that both $\lesst$ and $\lesstrong$ are transitive by transitivity of $\dnf$.

We can define unions of towers in the natural way. In general a union of towers may not be a tower, but provided the union is long enough, it will also be a tower.

\begin{definition}\label{unions-of-towers}
    Let $\langle \calt^j : j < \alpha \rangle$ be a $\lesst$-increasing sequence of towers, where $\calt^j = \langle M_i^j : i \in I^j \rangle ^\wedge \langle a_i : i \in (I^j)^-\rangle$. Suppose that $\bigcup_{j < \alpha} I^j$ is a well ordered set. Define $\bigcup_{j < \alpha} \calt^j = \langle M_i^\alpha : i \in I^\alpha \rangle ^\wedge \langle a_i : i \in (I^\alpha)^-\rangle$ where
    \begin{enumerate}
        \item $I^\alpha = \bigcup_{j<\alpha} I^j$
        \item $M_i^\alpha = \bigcup_{j<\alpha} M_i^j$
    \end{enumerate}
\end{definition}

\begin{fact}[{\cite[3.22]{bema}}]\label{high-cofinality-unions-are-towers}
    Let $\langle \calt^j : j < \alpha \rangle$ be a $\lesst$-increasing sequence of towers, where $\calt^j = \langle M_i^j : i \in I^j \rangle ^\wedge \langle a_i : i \in (I^j)^-\rangle$. Suppose that $\bigcup_{j < \alpha} I^j$ is a well ordered set. Then $\bigcup_{j<\alpha} \calt^j$ is a tower, and for all $j < \alpha$, $\calt^j \lesst \bigcup_{j<\alpha} \calt^j$.
\end{fact}

We can always extend towers, and even arbitrary $\lesst$-increasing chains of towers, with a larger tower.

\begin{fact}\label{tower-chain-extensions}
    Let $\langle \calt^j : j < \alpha \rangle$ be a $\lesst$-increasing sequence of towers, where $\calt^j$ is indexed by $I^j$. Suppose that $\bigcup_{j < \alpha} I^j$ is a well ordered set. Then there exists a strongly $\underhlim$-increasing tower $\calt^\alpha$ indexed by $\bigcup_{j<\alpha} I^j$ such that $\calt^j \lesstrong \calt^\alpha$ for all $j < \alpha$.

    Moreover, if $\calt^j = \langle M_i^j : i \in I^j \rangle ^\wedge \langle a_i : i \in (I^j)^-\rangle$ for $j \leq \alpha$, then for all $i \in I^\alpha$, $\bigcup\{M_i^j : j < \alpha, i \in I^j\} \underhlim M_i^\alpha$.
\end{fact}

\begin{proof}
    The first paragraph is {\cite[3.24]{bema}}. The moreover part is inherent in the construction of {\cite[3.24]{bema}}; or, if you prefer, take strongly $\underhlim$-increasing $\calt^{\alpha + 1}$ such that $\calt^\alpha \lesstrong \calt^{\alpha+1}$ using the first part. Then $\bigcup\{M_i^j : j < \alpha, i \in I^j\} \lek M_i^\alpha \underhlim M_i^{\alpha + 1}$, so you can replace $\calt^\alpha$ with $\calt^{\alpha+1}$.
\end{proof}

Note the $\alpha = 1$ case says we can extend a single tower.

Sometimes we need to ensure that a non-forking type is realised at the bottom of a tower extension.

\begin{fact}[{\cite[3.25]{bema}}]\label{tower-extensions-with-b}
    Let $\calt = \langle M_i : i \in I \rangle ^\wedge \langle a_i : i \in I^-\rangle$ be a tower. Let $p \in \gS(M_0)$. Then there exists a strongly $\underhlim$-increasing tower $\calt' = \langle M_i' : i \in I \rangle ^\wedge \langle a_i : i \in I^-\rangle$ and $b \in M_0'$ such that $\calt \lesstrong \calt'$, $p = \gtp(b/M_0, M_0')$, and $\gtp(b/M_i, M_i')$ $\dnf$-does not fork over $M_0$ for all $i \in I$.
\end{fact}

Reduced towers are towers that increase gradually enough that intersecting high models with low models from an extension gives the smallest model possible from the smaller tower.

\begin{definition}\label{reduced-definition}
    A tower $\calt = \langle M_i : i \in I \rangle ^\wedge \langle a_i : i \in I^-\rangle$ is \emph{reduced} if for all towers $\calt' = \langle M_i' : i \in I' \rangle ^\wedge \langle a_i : i \in (I')^-$ such that $\calt \lesst \calt'$, for all $r < i \in I$, we have $M_i \cap M_r' = M_r$.
\end{definition}

Reduced towers are a central concept for our main argument - this `smallest intersection' property will be what ensures our amalgamation is disjoint.

Unions of high cofinality chains of reduced towers are also reduced.

\begin{fact}[{\cite[3.30]{bema}}]\label{unions-of-reduced-are-reduced}
    If $\delta < \lambda^+$ and $\langle \calt^j : j < \delta \rangle$ is a $\lesst$-increasing chain of reduced towers, and $\cof(\delta) \geq \kappa$, then $\bigcup_{j < \delta} \calt^\delta$ is reduced.
\end{fact}

We can always find a reduced extension of any tower. In fact, by building a long chain of reduced extensions, the union will be a $\lesstrong$-extension.

\begin{lemma}\label{reduced-extensions}
    Let $\calt$ be a tower indexed by $I$. Then there is a reduced tower $\calt'$ indexed by $I$ such that $\calt \lesstrong \calt'$
\end{lemma}

\begin{proof}
    By \cite[3.31]{bema}, we may find a reduced tower $\lesst$-extending any other tower with the same index. Form a $\lesst$-increasing chain of length $\kappa$ of towers $\langle\calt^j : j < \kappa \rangle$ such that $\calt^0 = \calt$ and $\calt^{j+1}$ is reduced for all $j < \kappa$. This is possible by using \cite[3.31]{bema} at successors and Fact \ref{tower-chain-extensions} at limits. Finally, take $\calt^\kappa = \bigcup_{j < \kappa} \calt$. Say $\calt^j = \langle M_i^j : i \in I \rangle ^\wedge \langle a_i : i \in I^-\rangle$ for $j \leq \kappa$. We have that $\calt^\kappa$ is reduced by Fact \ref{unions-of-reduced-are-reduced}, $\calt \lesst \calt^\kappa$ by Fact \ref{unions-of-towers}, and furthermore $\calt \lesstrong \calt'$ since $\langle M_i^j : j < \kappa \rangle$ witnesses that $M_i^\kappa$ is a $(\lambda, \kappa)$-limit model over $M_i^0$. Thus $\calt' = \calt^\kappa$ fulfils our requirements.
\end{proof}

The sequence of models of a reduced tower is continuous at all high cofinality $i \in I$. This was central in the main argument of \cite[\textsection 3]{bema} to create a tower that contained a $(\lambda, \geq\kappa)$-limit model as well as it's witnessing sequence. This idea will come up again in Lemma \ref{building-a-reduced-tower-between-m0-m1}.

\begin{fact}[{\cite[3.35]{bema}}]\label{reduced-implies-high-continuity}
    If $\calt$ is a reduced tower indexed by $I$, and $i \in I$ has $\cof_I(i) \geq \kappa$, then $\calt$ is continuous at $i$.
\end{fact}

Full towers are towers that realise non-forking extensions of all types over certain models at regular intervals. By realising enough such types, using Fact \ref{universal-and-limit-extensions-exist} we can ensure a subsequence of models in the tower is universally increasing.

\begin{definition}\label{full-definition}
    Let $\calt = \langle M_i : i \in I \rangle ^\wedge \langle a_i : i \in I^-\rangle$ be a tower and $I_0 \subseteq I$. $\calt$ is \emph{$I_0$-full} if for every $i \in (I_0)^-$ and every $p \in gS(M_i)$, there exists $k \in [i, i+_{I_0}1)_I$ such that $\gtp(a_k/M_k, M_{k+_I1})$ extends $p$ and $\dnf$-does not fork over $M_i$.
\end{definition}

In a strongly $\lek^u$-increasing tower of the correct shape, we can insert intermediate models in such a way that the tower becomes full.

\begin{fact}[{\cite[3.40]{bema}}]\label{full-extensions}
    Let $I$ be a well order, $\alpha<\lambda^+$ a limit ordinal, and $\gamma \in (\alpha, \lambda^+)$ a limit ordinal with $\cof(\gamma) = \lambda$. Suppose $\calt = \langle M_i : i \in I \times \alpha \rangle ^\wedge \langle a_i : i \in I\times \alpha\rangle$ is a strongly $\underhlim$-tower. Then there exists a $I \times \{0\}$-full tower $\calt'$ indexed by $I \times \gamma$ such that $\calt' \upharpoonright (I \times \alpha) = \calt$.
\end{fact}

\begin{remark}
	Again, in removing the $a_i \in |M_{i+1}| \setminus |M_i|$ assumption, the proofs from \cite[3.40]{bema} go through allowing algebraic types.
\end{remark}

Like reduced towers, full towers behave well under unions.

\begin{fact}[{\cite[3.38]{bema}}]\label{full-unions}
    Let $I_0 \subseteq I$ be well orders. If $\langle \calt^j : j < \delta \rangle$ is a sequence of towers indexed by $I$ which are $I_0$ full and $\cof(\delta) \geq \kappa$, then $\bigcup_{j < \delta} \calt^j$ is an $I_0$-full tower.
\end{fact}

Using the facts above, one can carry out the main argument of \cite[\textsection 3]{bema}. We emphasise the hypotheses again.

\begin{fact}[{\cite[3.1]{bema}}]\label{bemain}
    Assume Hypothesis \ref{general-hypothesis}. Suppose $\delta_1, \delta_2 < \lambda^+$ are limit ordinals with $\cof(\delta_1), \cof(\delta_2) \geq \kappa$. If $M, N_1, N_2 \in \K_\lambda$ such that $N_l$ is a $(\lambda, \delta_l)$-limit model over $M$ for $l = 1, 2$, then there exists an isomorphism $f : N_1 \cong N_2$ fixing $M$.

    Furthermore, if $N_1, N_2 \in \K_\lambda$ and $N_l$ is a $(\lambda, \delta_l)$-limit model for $l = 1, 2$, then $N_1 \cong N_2$.
\end{fact}

\section{Main results}

The goal of this section is to prove the main theorem:

\begin{restatable}{theorem}{maintheorem}\label{disjoint-nf-amalgamation}
	Assume Hypothesis \ref{general-hypothesis}. Assume $M_0, M_1, M_2 \in \Kkappalims$, and that $M_0 \lek M_l$ and $a_l \in M_l$ for $l = 1, 2$. Then there exist $N \in \Kkappalims$ and $f_l : M_l \rightarrow N$ fixing $M_0$ for $l = 1, 2$ such that $\gtp(f_l(a_l)/f_{3-l}[M_{3-l}], N)$ $\dnf$-does not fork over $M_0$ and $f_1[M_1] \cap f_2[M_2] = M_0$. That is, $\dnf \upharpoonright \Kkappalims$ satisfies disjoint non-forking amalgamation.
	
	In particular, every $M_0 \in \Kkappalims$ is a disjoint amalgamation base in $\K_\lambda$.
\end{restatable}

Our argument is loosely inspired by the proof of \cite[5.36]{vas19}. The overall strategy is as follows: we will show that we can form a reduced tower $\calt^1$ with $M_0$ at the bottom and $M_1$ at the top such that the first singleton is $a_1$. We then `amalgamate' this with the tower $\calt^2 = \langle M_0, M_2 \rangle ^\wedge \langle a_2 \rangle$. By `amalgamate', we mean that we find a matrix of models where the leftmost column is given by $\calt^1$ and the bottom row is given by the image of $\calt^2$ under some map $f$ (see Figure \ref{disjoint-ap-main-construction}). In fact, calling the second column $\calt^*$, we will have $\calt^1 \lesstrong \calt^*$. The type of $a_1$ over $f[M_2]$ will not fork over $M_0$ because of the tower ordering, and by amalgamating the towers carefully using a modified version of Fact \ref{tower-extensions-with-b} (Lemma \ref{successor-step-for-tower-ap}), we can make sure that the type of $f(a_2)$ over $M_1$ $\dnf$-does not fork over $M_0$ also. So taking $N = M_{\alpha, 1}$, $f_1:M_1 \rightarrow N$ the inclusion map, and $f_2 = f:M_2 \rightarrow N$, we have the desired properties. In particular, disjointness follows from making the first tower reduced (see Figure \ref{disjoint-ap-main-construction}).

\begin{figure}[!ht]
\centering

\begin{circuitikz}
\tikzstyle{every node}=[font=\normalsize]
\node [font=\normalsize] at (5.75,5.75) {$M_{0, 0}$};
\draw [->, >=Stealth] (5.75,6) -- (5.75,7.25);
\draw [->, >=Stealth] (6.25,5.75) -- (7.5,5.75);
\draw [->, >=Stealth] (6.25,7.75) -- (7.5,7.75);
\draw [->, >=Stealth] (5.75,8) -- (5.75,9.25);
\node [font=\normalsize] at (5.75,7.75) {$M_{1, 0}$};
\draw [->, >=Stealth] (6.25,9.75) -- (7.5,9.75);
\draw [->, >=Stealth] (5.75,10) -- (5.75,11.25);
\node [font=\normalsize] at (5.75,9.75) {$M_{2, 0}$};
\node [font=\normalsize] at (4.75,5.75) {$M_0$};
\node [font=\normalsize] at (5.75,11.75) {$\vdots$};
\node [font=\normalsize] at (4.75,12.5) {$M_1$};
\node [font=\normalsize] at (5.75,12.5) {$M_{\alpha, 0}$};
\node [font=\normalsize] at (8,12.5) {$M_{\alpha, 1}$};
\node [font=\normalsize] at (8,11.75) {$\vdots$};
\draw [->, >=Stealth] (8,10) -- (8,11.25);
\node [font=\normalsize] at (8,9.75) {$M_{2, 1}$};
\draw [->, >=Stealth] (8,8) -- (8,9.25);
\node [font=\normalsize] at (8,7.75) {$M_{1, 1}$};
\draw [->, >=Stealth] (8,6) -- (8,7.25);
\node [font=\normalsize] at (8,5.75) {$M_{0, 1}$};
\draw [->, >=Stealth] (6.25,12.5) -- (7.5,12.5);
\node [font=\normalsize] at (5.75,13.25) {$\calt^1$};
\node [font=\normalsize] at (6.9,13.15) {$\lesstrong$};
\node [font=\normalsize] at (8,13.25) {$\calt^*$};
\node [font=\normalsize] at (5.2,12.5) {$=$};
\node [font=\normalsize] at (5.2,5.75) {$=$};
\node [font=\normalsize] at (5.75,4.75) {$M_0$};
\node [font=\normalsize] at (8,4.75) {$M_2$};
\node [font=\normalsize] at (9.5,5.75) {`$f[\calt^2]$'};
\node [font=\normalsize] at (9.5,4.75) {$\calt^2$};
\node [font=\normalsize, rotate around={90:(0,0)}] at (5.75,5.25) {$=$};
\node [font=\normalsize, rotate around={-90:(0,0)}] at (8,5.25) {$\cong$};
\draw [->, >=Stealth] (6.25,4.75) -- (7.5,4.75);
\node [font=\normalsize] at (8.5,5.25) {$f$};
\node [font=\normalsize] at (5,7.75) {$a_1 \in$};
\node [font=\normalsize] at (8,4) {$a_2$};
\node [font=\normalsize, rotate around={90:(0,0)}] at (8,4.25) {$\in$};
\end{circuitikz}

\caption{The main construction}
\label{disjoint-ap-main-construction}
\end{figure}

The brunt of the work when applying this strategy lies in building the tower $\calt^1$ (in particular, ensuring it is reduced), and making sure the amalgamation of towers has the desired non-forking properties for $a_2$. We can achieve this by mimicking the overall approach from \cite{vas19} and by using the tools already developed for towers in this setting from \cite{bema}.

\begin{remark}
    It is worth making some comparison with \cite[5.36]{vas19}, in particular where our approach and theirs differ.
    \begin{enumerate}
        \item Building the first tower $\calt^1$ with the necessary properties (in particular, being reduced) is more difficult in our context - we summarise the problems and our workarounds at the start of Subsection \ref{building-a-reduced-tower-between-m0-m1}.
        \item \cite{vas19} assumes full categoricity, but while not stated explicitly, it appears the only reason this is used rather than the weaker $\lambda$-superstable $\lambda$-symmetric assumptions of the rest of \cite[\textsection 5]{vas19} is to get that \emph{every} $M \in \K_\lambda$ is an amalgamation base (rather than just $(\lambda, \geq \aleph_0)$-limit models). This was the basis of why Theorem \ref{disjoint-nf-amalgamation} seemed plausible to the author in our more general setting.
        \item In the proof of \cite[5.36]{vas19}, the second tower $\calt'$ ($\calt^2$ in our terminology) is with extra work made to be reduced and of some limit length $\delta^*$, witnessing that $M_2$ is a $(\lambda, \geq \aleph_0)$-limit model over $M_0$. However, since the proof only uses that $\calt$ is reduced, you could instead take $\calt' = \langle M_0, M_2 \rangle ^\wedge \langle a_2 \rangle$ as in our argument, which simplifies the picture.
    \end{enumerate}
\end{remark}

\subsection{Tower amalgamation}\label{tower-amalgamation}

Throughout this subsection, assume Hypothesis \ref{proof-hypothesis} holds for some $\K$ stable in $\lambda \geq \LS(\K)$, $\kappa < \lambda^+$ regular, and $\dnf$ an independence relation on $\Kkappalims$.

We will use the following lemma fairly frequently and without mention.

\begin{lemma}
    If $M \underhlim N$, then there exists $N_0$ such that $M \underhlim N_0 \underhlim N$.
    
    Moreover, if $b \in N$, then $N_0$ can be chosen such that $b \in N_0$.
\end{lemma}

\begin{proof}
    Build a $\underhlim$-increasing sequence $\langle M_0 : i \leq \kappa \rangle$ continuous at $\kappa$ such that $M_0 = M$. By Fact \ref{bemain}, without loss of generality $M_\kappa = N$. Then take $N_0 = M_1$.
    
    For the moreover part, instead take $N_0 = M_i$ for some $i<\kappa$ large enough that $b \in M_i$.
\end{proof}

\begin{lemma}\label{type-ap-lim-isomorphism-lemma}
    Suppose that $M_0, M_1, M_2 \in \K_\lambda$, and $b \in M_1, c \in M_2$, where $M_0 \underhlim M_1$, $M_0 \underhlim M_2$, and $\gtp(b/M_0, M_1) = \gtp(c/M_0, M_2)$. Then there exists an isomorphism $f : M_1 \cong M_2$ fixing $M_0$ such that $f(b) = c$.
\end{lemma}

\begin{proof}
    Since $M_0 \underhlim M_1$, there exists $M^0_1$ such that $M_0 \underhlim M^0_1 \underhlim M_1$ and $b \in M^0_1$. Similarly, as $M_0 \underhlim M_2$ there exists $M^0_2$ such that $M_0 \underhlim M^0_2 \underhlim M_2$ and $c \in M^0_2$.

    By type equality, there exist $N^0 \in \K_\lambda$ such that $M^0_2 \lek N^0$ and $h_0 : M_1^0 \rightarrow N^0$ such that $h_0(b) = c$. Take some $M_1^1 \in \K$ and $h_1 : M_1^1 \cong N^0$ an isomorphism extending $h_0$. Since $M_0^1 \underhlim M$, without loss of generality we can assume $M_0^1 \lek M_1^1 \underhlim M$ (`push' $M_1^1$ under the first model in a sequence witnessing $M_0^1 \underhlim M$). Now extend $h_1$ to an isomorphism $h : M_1 \cong N$ for some $N \in \K$. Note that $N$ and $M_2$ are both $(\lambda, \geq \kappa)$-limit models over $M_2^0$, so there exists $g:N \cong M_2$ fixing $M_2^0$ by Fact \ref{bemain}. We have constructed the system in Figure \ref{tp-ap-lim-isomorphism}.

    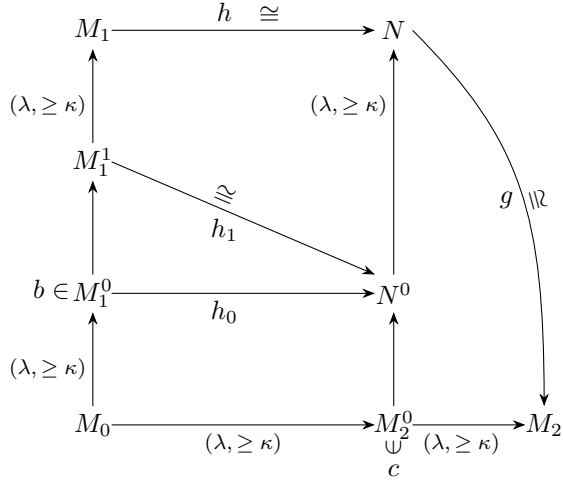
\begin{figure}[!ht]
\centering

\begin{circuitikz}
\tikzstyle{every node}=[font=\normalsize]
\node [font=\normalsize] at (3,8) {$M_0$};
\draw [->, >=Stealth] (3,8.25) -- (3,9.5);
\draw [->, >=Stealth] (3.25,8) -- (6.75,8);
\node [font=\normalsize] at (3,9.75) {$M_1^0$};
\draw [->, >=Stealth] (3,10) -- (3,11.25);
\draw [->, >=Stealth] (3.25,9.75) -- (6.75,9.75);
\node [font=\normalsize] at (3,11.5) {$M_1^1$};
\draw [->, >=Stealth] (3,11.75) -- (3,13);
\draw [->, >=Stealth] (3.25,11.5) -- (6.75,10);
\node [font=\normalsize] at (7,8) {$M_2^0$};
\draw [->, >=Stealth] (7,8.25) -- (7,9.5);
\draw [->, >=Stealth] (7.25,8) -- (8.75,8);
\node [font=\normalsize] at (7,9.75) {$N^0$};
\draw [->, >=Stealth] (7,10) -- (7,13);
\node [font=\normalsize] at (3,13.25) {$M_1$};
\node [font=\normalsize] at (9,8) {$M_2$};
\draw [->, >=Stealth] (3.25,13.25) -- (6.75,13.25);
\node [font=\normalsize] at (7,13.25) {$N$};
\draw [->, >=Stealth] (7.25,13.25) .. controls (8.75,11.75) and (9,10.75) .. (9,8.25) ;
\node [font=\normalsize] at (2.45,9.8) {$b \in$};
\node [font=\normalsize] at (7,7.4) {$c$};
\node [font=\normalsize, rotate=90] at (7,7.7) {$\in$};
\node [font=\normalsize] at (4.75,9.5) {$h_0$};
\node [font=\normalsize] at (4.75,10.6) {$h_1$};
\node [font=\normalsize, rotate = 337] at (4.75,11.05) {$\cong$};
\node [font=\normalsize] at (4.75,13.5) {$h$};
\node [font=\normalsize] at (5.35,13.45) {$\cong$};
\node [font=\normalsize] at (8.5,11) {$g$};
\node [font=\normalsize, rotate=285] at (8.9,11) {$\cong$};
\node [font=\footnotesize] at (5,7.75) {$(\lambda, \geq \kappa)$};
\node [font=\footnotesize] at (7.9,7.75) {$(\lambda, \geq \kappa)$};
\node [font=\footnotesize] at (2.4,8.75) {$(\lambda, \geq \kappa)$};
\node [font=\footnotesize] at (2.4,12.25) {$(\lambda, \geq \kappa)$};
\node [font=\footnotesize] at (6.4,12.25) {$(\lambda, \geq \kappa)$};
\end{circuitikz}

\caption{The construction from Lemma \ref{type-ap-lim-isomorphism-lemma}.}

\label{tp-ap-lim-isomorphism}
\end{figure}

    Let $f = g \circ h : M_1 \cong M_2$. Since $g$ and $h$ both fix $M_0$, $f$ does also, and $f(b) = g(h(b)) = g(c) = c$ since $g$ fixes $M_2^0$.
\end{proof}

\begin{lemma}\label{successor-step-for-tower-ap}
    Let $\calt = \langle M_i : i < \alpha \rangle ^\wedge \langle a_i : i < \alpha^-\rangle$ be a tower, let $N_0$ be a $(\lambda, \geq \kappa)$-limit model over $M_0$, and let $b \in N_0$. Then there exist a tower $\calt' = \langle M_i' : i < \alpha \rangle ^\wedge \langle a_i : i < \alpha^-\rangle$ and an isomorphism $f:N_0 \cong M_0'$ fixing $M_0$ such that $\calt \lesst \calt'$ and $\gtp(f(b)/M_i, M_i') \dnf$-does not fork over $M_0$.
\end{lemma}

\begin{proof}
    Using Fact \ref{tower-extensions-with-b}, take some $\calt' = \langle M_i' : i < \alpha \rangle ^\wedge \langle a_i : i < \alpha^-\rangle$ with $b' \in M_0'$ such that $\calt \lesstrong \calt'$ and $\gtp(b'/M_i, M_i')$ $\dnf$-does not fork over $M_0$ and extends $\gtp(b/M_0, M_0')$ for all $i < \alpha$. By Lemma \ref{type-ap-lim-isomorphism-lemma} there exists $f : N_0 \cong M_0'$ fixing $M_0$ such that $f(b) = b'$, as required.
\end{proof}

The following generalises \cite[5.32]{vas19}. We will only need it in the case where $\beta = 2$ when proving Theorem \ref{disjoint-nf-amalgamation}. While restricting to $\beta = 2$ would arguably simplify the presentation of the proof a little, we will use arbitrary $\beta$ as this may be useful in other contexts.

\begin{proposition}[Tower amalgamation]\label{tower-ap}
    Given towers $\calt = \langle M_i : i < \alpha \rangle ^\wedge \langle a_i : i < \alpha^-\rangle$ and $\calt' = \langle N_j : j < \beta \rangle ^\wedge \langle b_j : j < \beta^-\rangle$ such that $M_0 = N_0$ and $\calt'$ is $\underhlim$-strongly increasing, there exist $\langle M_{i, j} : i < \alpha, j < \beta \rangle$ in $\Kkappalims$ and $\langle f_j : j < \beta\rangle$ an increasing sequence of isomorphisms $f_j : N_j \cong M_{0, j}$ such that
    \begin{enumerate}
        \item $M_{i, 0} = M_i$ and $f_0 = \operatorname{id}_{N_0}$
        \item for all $i < \alpha$, $\langle M_{i, j}:j < \beta \rangle$ is $\underhlim$-increasing
        \item for all $j < \beta$, $\langle M_{i, j} : i < \alpha \rangle$ is $\lek$-increasing
        \item for all $i < \alpha^-$ and $j < \beta$, $\gtp(a_i/M_{i, j}, M_{i+1, j})$ $\dnf$-does not fork over $M_{i, 0}$
        \item for all $j < \beta^-$ and $i < \alpha$, $\gtp(f_{j+1}(b_j)/M_{i, j}, M_{i, j+1})$ $\dnf$-does not fork over $M_{0, j}$
    \end{enumerate}

    Furthermore, if $\calt$ is reduced, then we can also ensure that for every $i < k < \alpha$ and every $j < \beta$, $M_{i, j} \cap M_{k, 0} = M_{i, 0}$.
\end{proposition}

\begin{proof}
	We build by induction a $\lesstrong$-increasing sequence $\calt^{(j)}$ such that 
    \begin{enumerate}
        \item $\calt^{(0)} = \calt$
        \item $\calt^{(j)} = \langle M_{i, j} : i < \alpha \rangle ^\wedge \langle a_i : i < \alpha^-\rangle$ for all $j < \beta$
        \item for all $j < \beta$, $f_j : N_j \cong M_{0, j}$ fixing $N_0$
        \item for all $j < \beta^-$ and $i < \alpha$ we have $\gtp(f_{j+1}(b_j)/M_{i, j}, M_{i, j+1})$ $\dnf$-does not fork over $M_{0, j}$.
    \end{enumerate}

    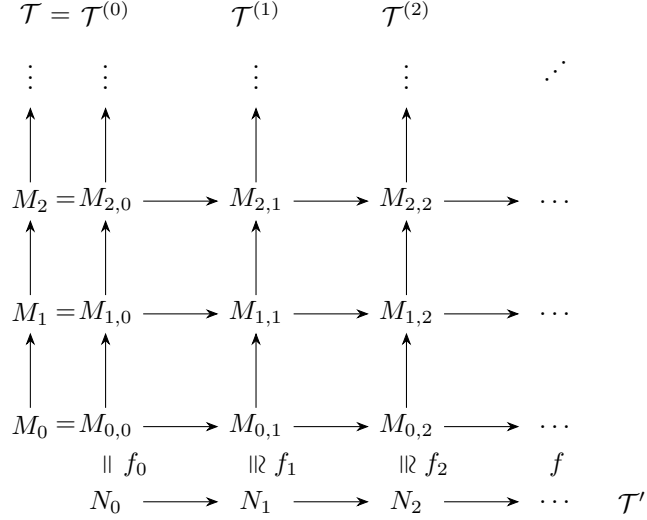
\begin{figure}[!ht]
\centering

\begin{circuitikz}
\tikzstyle{every node}=[font=\normalsize]
\node [font=\normalsize] at (4,5.75) {$M_{0, 0}$};
\draw [->, >=Stealth] (4,6) -- (4,7);
\draw [->, >=Stealth] (4.5,5.75) -- (5.5,5.75);
\draw [->, >=Stealth] (6,6) -- (6,7);
\draw [->, >=Stealth] (6.5,5.75) -- (7.5,5.75);
\node [font=\normalsize] at (6,5.75) {$M_{0, 1}$};
\draw [->, >=Stealth] (8,6) -- (8,7);
\draw [->, >=Stealth] (8.5,5.75) -- (9.5,5.75);
\node [font=\normalsize] at (8,5.75) {$M_{0, 2}$};
\draw [->, >=Stealth] (8,7.5) -- (8,8.5);
\draw [->, >=Stealth] (8.5,7.25) -- (9.5,7.25);
\node [font=\normalsize] at (8,7.25) {$M_{1, 2}$};
\draw [->, >=Stealth] (4,7.5) -- (4,8.5);
\draw [->, >=Stealth] (6,7.5) -- (6,8.5);
\node [font=\normalsize] at (4,7.25) {$M_{1, 0}$};
\draw [->, >=Stealth] (4.5,7.25) -- (5.5,7.25);
\node [font=\normalsize] at (6,7.25) {$M_{1, 1}$};
\draw [->, >=Stealth] (6.5,7.25) -- (7.5,7.25);
\draw [->, >=Stealth] (8,9) -- (8,10);
\draw [->, >=Stealth] (8.5,8.75) -- (9.5,8.75);
\node [font=\normalsize] at (8,8.75) {$M_{2, 2}$};
\draw [->, >=Stealth] (4,9) -- (4,10);
\draw [->, >=Stealth] (6,9) -- (6,10);
\node [font=\normalsize] at (4,8.75) {$M_{2, 0}$};
\draw [->, >=Stealth] (4.5,8.75) -- (5.5,8.75);
\node [font=\normalsize] at (6,8.75) {$M_{2, 1}$};
\draw [->, >=Stealth] (6.5,8.75) -- (7.5,8.75);
\node [font=\normalsize] at (3,5.75) {$M_0$};
\node [font=\normalsize] at (3,7.25) {$M_1$};
\draw [->, >=Stealth] (3,6) -- (3,7);
\draw [->, >=Stealth] (3,7.5) -- (3,8.5);
\node [font=\normalsize] at (3,8.75) {$M_2$};
\draw [->, >=Stealth] (3,9) -- (3,10);
\node [font=\normalsize] at (3,10.5) {$\vdots$};
\node [font=\normalsize] at (4,10.5) {$\vdots$};
\node [font=\normalsize] at (6,10.5) {$\vdots$};
\node [font=\normalsize] at (8,10.5) {$\vdots$};
\node [font=\normalsize, rotate=45] at (10,10.5) {$\dots$};
\node [font=\normalsize] at (10,7.25) {$\dots$};
\node [font=\normalsize] at (10,5.75) {$\dots$};
\node [font=\normalsize] at (10,8.75) {$\dots$};
\node [font=\normalsize, rotate=270] at (4,5.25) {$=$};
\node [font=\normalsize] at (4.4,5.25) {$f_0$};
\node [font=\normalsize, rotate=270] at (6,5.25) {$\cong$};
\node [font=\normalsize] at (6.4,5.25) {$f_1$};
\node [font=\normalsize, rotate=270] at (8,5.25) {$\cong$};
\node [font=\normalsize] at (8.4,5.25) {$f_2$};
\node [font=\normalsize] at (10,5.25) {$f$};
\node [font=\normalsize] at (4,4.75) {$N_0$};
\node [font=\normalsize] at (6,4.75) {$N_1$};
\node [font=\normalsize] at (8,4.75) {$N_2$};
\draw [->, >=Stealth] (4.5,4.75) -- (5.5,4.75);
\draw [->, >=Stealth] (6.5,4.75) -- (7.5,4.75);
\draw [->, >=Stealth] (8.5,4.75) -- (9.5,4.75);
\node [font=\normalsize] at (10,4.75) {$\dots$};
\node [font=\normalsize] at (4,11.25) {$\calt^{(0)}$};
\node [font=\normalsize] at (6,11.25) {$\calt^{(1)}$};
\node [font=\normalsize] at (8,11.25) {$\calt^{(2)}$};
\node [font=\normalsize] at (11,4.75) {$\calt'$};
\node [font=\normalsize] at (3,11.25) {$\calt$};
\node [font=\normalsize] at (3.47,8.75) {$=$};
\node [font=\normalsize] at (3.47,7.25) {$=$};
\node [font=\normalsize] at (3.47,5.75) {$=$};
\node [font=\normalsize] at (3.4,11.2) {$=$};
\end{circuitikz}

\caption{The tower amalgamation construction in Proposition \ref{tower-ap}.}

\label{tower-ap-construction}
\end{figure}

    This is enough to guarantee the desired conditions hold. In particular, if $\calt = \calt^{(0)}$ is reduced, for every $j<\beta$ we have that for every $i < k < \alpha$, $M_{i, j} \cap M_{k, 0} = M_{i, 0}$ since $\calt^{(0)} \lesst \calt^{(j)}$.

    Now we proceed to the construction. For $j = 0$, $\calt^{(0)} = \calt$ and $f_0 = \operatorname{id}_{N_0}$ are given. For limit $j < \beta$, by Fact \ref{tower-chain-extensions}, we may take $\calt^{(j)}$ to be any $\lesstrong$ extension of all $\langle \calt^{(k)} : k < j \rangle$ such that $\bigcup_{k < j} M_{i, k} \underhlim M_{i, j}$ for all $i < \alpha$. Since $\bigcup_{k<j} N_k \underhlim N_j$, $\bigcup_{k<j} M_{0, k} \underhlim M_{0, j}$, and $\bigcup_{k < j} f_k : \bigcup_{k<j} N_k \cong \bigcup_{k<j} M_{0, k}$, by Fact \ref{bemain} there is an isomorphism $f_j : N_j \cong M_{0, j}$ extending $\bigcup_{k < j} f_k$, which still fixes $M_{0, 0}$.

    For successors, suppose $j < \beta^-$, and $\calt^{(j)}$ and $f_j$ are defined. Take some $\bar{M}_{0, j+1}$ and an isomorphism $\bar{f}_{j+1} : N_{j+1} \cong \bar{M}_{0, j+1}$ extending $f_j$. By applying Lemma \ref{successor-step-for-tower-ap} to $\calt^{(j)}$ and $M_{0, j} \underhlim \bar{M}_{0, j+1}$, there exist $M_{0, j+1} \in \K$ and some $g_j : \bar{M}_{0, j+1} \cong M_{0, j+1}$ fixing $M_{0, j}$ such that $\gtp(g_j(\bar{f}_{j+1}(b_j))/M_{i, j}, M_{i, j+1})$ $\dnf$-does not fork over $M_{0, j}$. Finally, take $f_{j+1} = g_j \circ \bar{f}_{j+1}$. Since $\bar{f}_{j+1}$ extends $f_j$ and $g$ fixes $M_{0, j}$, $f_{j+1} \supseteq f_j$. This completes the construction, and the proof.
\end{proof}

\begin{remark}
	The statement holds true for $\calt'$ not strongly $\underhlim$-increasing, provided that we weaken condition (2) to hold only for $i \in [1, \alpha)$. To see this, take $\calt'' = \langle N_j'' : i < \beta \rangle ^\wedge \langle b_j : j < \beta^-\rangle$ where $\calt' \lesst \calt''$ and $\calt''$ is strongly $\underhlim$-increasing by Fact \ref{tower-chain-extensions}. Let $\calt''_0$ be $\calt''$ but with the bottom model replaced by $M_0 = N_0$. $\calt''_0$ is still strongly $\underhlim$-increasing, so we can apply the theorem to $\calt$ and $\calt''$. Replacing $M_{0, j}$ by $f_j[N_j]$ and $f_j$ by $f_j \upharpoonright N_j$, the resulting matrix $\langle M_{i, j} : i < \alpha, j < \beta \rangle$ and function $\langle f_j : j < \beta \rangle$ work. It is not difficult to see that conditions (1), (3), and (4) still hold. (2) still holds for $i \neq 0$ as we have not altered the $M_{i, j}$ for $i \neq 0$. Finally, since $\gtp(b_j/N_j'', N_{j+1}'')$ $\dnf$-does not fork over $N_j$, $\gtp(f_{j+1}(b_j)/M_{0, j}, M_{0, j+1})$ $\dnf$-does not fork over $f_j[N_j]$, so transitivity guarantees that $\gtp(f_{j+1}(b_j)/M_{i, j}, M_{i, j+1})$ $\dnf$-does not fork over $f_j[N_j]$, meaning that condition (5) of the Proposition statement is preserved. The `further' part still holds since $M_0 = M_{0, 0}$ is still a substructure of $f_j[N_j]$ since each $f_j$ extends $\operatorname{id}_{N_0}$.
\end{remark}

\subsection{Building a reduced tower between $M_0$ and $M_1$}\label{building-a-reduced-tower-between-m0-m1}

Throughout this subsection, assume Hypothesis \ref{proof-hypothesis} holds for some $\K$ stable in $\lambda \geq \LS(\K)$, $\kappa < \lambda^+$ regular, and $\dnf$ an independence relation on $\Kkappalims$.

The argument of \cite[Lemma 5.36]{vas19} often makes use of the fact that $\calt = \langle M_i : i < \alpha \rangle ^\wedge \langle a_i : i < \alpha^- \rangle$ is reduced if and only if $\calt$ is continuous and $\langle M_i, M_{i+1}\rangle^\wedge \langle a_i \rangle$ is reduced for all $i < \alpha^-$. In our setting we lose this equivalence since towers may not be continuous at all limits, so we have a harder time building the reduced tower $\calt$ with $M_0$ at the bottom, $M_1$ at the top, and with first singleton $a_1$. We get around the issue in this Subsection by
\begin{enumerate}
    \item finding a 2-indexed reduced tower with $M_0$ at the bottom and singleton $a_1$ (Lemma \ref{mini-reduced-under-N})
    \item mimicking the proof of \cite[3.1]{bema} to build a reduced tower that is `full' enough to contain a sequence witnessing that the top element is a $(\lambda, \geq \kappa)$-limit model, which we can make into $M_1$ using Fact \ref{bemain} (the main construction in Proposition \ref{reduced-tower-between-limit-models})
    \item `stitching' these two towers together to get a reduced tower of the desired form (Lemma \ref{reduced-stitching}).
\end{enumerate}

\begin{lemma}\label{reduced-with-M-on-bottom}
    For any $M \in \Kkappalims$ and $p \in \gS(M)$, there exists $N \in \Kkappalims$ and $a \in N$ such that $\langle M, N \rangle ^\wedge \langle a \rangle$ is reduced and $\gtp(a/M, N) = p$.
\end{lemma}

\begin{proof}
    Say $p = \gtp(b/M, N_0)$, where $N_0 \in \Kkappalims$. By Fact \ref{reduced-extensions}, there exists a reduced tower $\langle M', N' \rangle ^\wedge \langle b \rangle$ with $\langle M, N_0 \rangle ^\wedge \langle b \rangle \lesstrong \langle M', N' \rangle ^\wedge \langle b \rangle$.
    
    Using local character on a $\underhlim$-increasing sequence witnessing $M \in \Kkappalims$, there is $M_0 \underhlim M$ such that $p$ $\dnf$-does not fork over $M_0$. Since then $M_0 \underhlim M'$ also, there is an isomorphism $f:M \cong M'$ fixing $M_0$ by Fact \ref{bemain}. Extend this to an isomorphism $g:N \cong N'$ for some model $N$, and take $a = g^{-1}(b)$. Note $N \in \K$ and $M \lek N$. Then the tower $\langle M, N \rangle ^\wedge \langle a \rangle$ is the image of the tower $\langle M', N' \rangle ^\wedge \langle b \rangle$ under $g^{-1}$, so it is reduced. 
    
    As $\langle M, N_0 \rangle ^\wedge \langle b \rangle \lesstrong \langle M', N' \rangle ^\wedge \langle b \rangle$, $\gtp(b/M', N')$ $\dnf$-does not fork over $M$. Since $\gtp(b/M, N)$ $\dnf$-does not fork over $M_0$, by transitivity $\gtp(b/M', N')$ $\dnf$-does not fork over $M_0$. Now, $\gtp(a/M, N) = g^{-1}(\gtp(b/M', N'))$, which (since $g$ fixes $M_0$) is an extension of $p \upharpoonright M_0$ that $\dnf$-does not fork over $M_0$ by invariance. Since this is true for $p$ also, by uniqueness of $\dnf$, we have that $\gtp(a/M, N) = p$ as desired.
\end{proof}

\begin{lemma}\label{mini-reduced-under-N}
    If $M \underhlim N$ and $b \in N$, then there exists $N_0\in \Kkappalims$ such that $M \lek N_0 \underhlim N$ and $\langle M, N_0 \rangle ^\wedge \langle b \rangle$ is reduced.
\end{lemma}

\begin{proof}
    Let $\langle N_i^* : i < \kappa \rangle$ be a strongly $\underhlim$-increasing sequence witnessing that $N$ is a $(\lambda, \kappa)$-limit over $M$, with $N_0^* = M$ (we can assume $N$ is a $(\lambda, \kappa)$-limit by Fact \ref{bemain}). By removing all $N_i^*$ for $i \in [1, \kappa)$ which do not contain $b$, we can assume without loss of generality that $b \in N_1^*$.

    By applying Lemma \ref{reduced-with-M-on-bottom} to $M$ and $\gtp(b/M, N)$, there is $N' \in \Kkappalims$ with $M \lek N'$ such that $\langle M, N' \rangle ^\wedge \langle a\rangle$ is reduced and $\gtp(b/M, N') = \gtp(a/M, N_1^*)$. By type equality, there is some $\bar{N} \in \Kkappalims$ where $N_1^* \lek \bar{N}$ and some $h : N' \rightarrow \bar{N}$ fixing $M$ such that $h(a) = b$. 
    
    Since $N_1^* \underhlim N_2^*$, there is $g : \bar{N} \rightarrow N_2^*$ fixing $N_1^*$. Let $f = g \circ h : N' \rightarrow N$. Then $\langle M, f[N'] \rangle ^\wedge \langle b \rangle$ is reduced (as the image of the reduced tower $\langle M, N'\rangle ^\wedge \langle a \rangle$ under $f$), and $M \lek f[N'] \lek N_2^*\underhlim N$, so $N_0 = f[N']$ is as desired.
\end{proof}

We will only use the following when $\beta = 1$, but we phrase as generally as possible. It essentially says that we may `glue' reduced towers on top of each other, provided they have some overlap, and the resultant tower will be reduced.

\begin{lemma}\label{reduced-stitching}
    Suppose $\calt = \langle M_i : i <\alpha \rangle ^\wedge \langle a_i : i <\alpha^-\rangle$ is a tower, where $\alpha \geq \beta$. If $\calt \upharpoonright (\beta+1)$ is reduced and $\calt \upharpoonright [\beta, \alpha)$ is reduced, then $\calt$ is reduced.
\end{lemma}

\begin{proof}
    Suppose $\calt \lesst \calt'$ for some tower $\calt' = \langle M_i' : i <\alpha \rangle ^\wedge \langle a_i : i <\alpha^-\rangle$. We must check that $M_j \cap M_i' = M_i$ for all $i < j < \alpha$. If $j \leq \beta$ then this follows from $\calt \upharpoonright (\beta+1)$ being reduced and $\calt\upharpoonright(\beta+1) \lesst \calt'\upharpoonright(\beta+1)$. If $i \geq \beta$ it follows from $\calt \upharpoonright [\beta, \alpha)$ being reduced and $\calt\upharpoonright[\beta, \alpha) \lesst \calt'\upharpoonright[\beta, \alpha)$. 
    
    So suppose $i < \beta < j$. Then $M_j \cap M_i' = M_j \cap M_\beta' \cap M_i' = M_\beta \cap M_i' = M_i$ (where the last two equalities follow from $\calt \upharpoonright [\beta, \alpha)$ and $\calt \upharpoonright (\beta + 1)$ being reduced respectively). In all cases $M_j \cap M_i' = M_i$, so $\calt$ is reduced as claimed.
\end{proof}

\begin{proposition}\label{reduced-tower-between-limit-models}
    Suppose $M_0 \underhlim M_1$ and $b \in M_1$. Then there exists $\alpha < \lambda^+$ a limit with $\cof(\alpha) = \kappa$ and a tower $\calt = \langle N_i : i \in I \rangle ^\wedge \langle a_i : i \in I^- \rangle$ for some well order $I$ with $\operatorname{otp}(I) = \gamma+1$ for some limit $\gamma$ with $\cof(\gamma) \geq \kappa$, where $I$ has minimal element $i_0$ and final element $i_1$, such that $\calt$ is reduced and $N_{i_0} = M_0$, $N_{i_1} = M_1$, and $b = a_{i_0}$.
\end{proposition}

\begin{proof}
    We mimic the approach of the proof of Fact \ref{bemain} (that is, \cite[3.1]{bema}) to build a tower that is reduced and $I_0$-full (for the appropriate ordering $I_0$). This will ensure the final model is a $(\lambda, \kappa)$-limit over the base model. By `stitching' $M_0$ to the bottom of the tower, we can turn the top model into $M_1$ since $M_0 \underhlim M_1$ while fixing $b$.

    By Lemma \ref{mini-reduced-under-N}, there exists $M^*\in \Kkappalims$ where $M \lek M^* \underhlim M_1$, $b \in M^*$, and $\langle M, M^* \rangle ^\wedge \langle b \rangle$ is reduced.

    For $\beta \geq \omega$, let $I_\beta = (\kappa + 1) \times \lambda \times \beta$ with the usual lexicographic ordering. Let $I_0 = (\kappa + 1) \times \lambda \times \{0\}$.

    By induction on $j \leq \kappa$, we define a continuous increasing sequence of limit ordinals $\langle \alpha_j : j \leq \kappa\rangle$ and a $\lesst$-increasing sequence of towers $\calt^j = \langle M_i^j : i \in I_{\alpha_j} \rangle ^\wedge \langle a_i : i \in I_{\alpha_j}^- \rangle$ that satisfy the following:
    \begin{enumerate}
        \item $\alpha_0 = \omega$
        \item for all $j < \kappa$, $\calt^{(2j+1)}$ is $I_0$-full
        \item for all $j < \kappa$, $\calt^{(2j + 2)}$ is reduced
        \item $\calt^\kappa = \bigcup_{j < \kappa} \calt^{j}$
    \end{enumerate}

    We describe the construction now: for $j = 0$, $\alpha_0$ is given, and $\calt^0$ can be constructed as in Remark \ref{towers-exist}. For $j < \kappa$ limit, take $\alpha_j = \bigcup_{k < j} \alpha_k$ and take any $\calt^j$ indexed by $I_{\alpha_j}$ extending all of $\langle \calt^k : k < j \rangle$ by Fact \ref{tower-chain-extensions}. For $j = \kappa$, $\alpha_\kappa$ and $\calt^\kappa$ are described in the requirements.

    Given $\alpha_{2j}$ and $\calt^{2j}$, we define $\alpha_{2j+1}$, $\calt^{2j+1}$, $\alpha_{2j+2}$ and $\calt^{2j+2}$. Take $\alpha_{2j+1} = \alpha_{2j} + \lambda$. Take any $\bar{\calt}^{2j}$ indexed by $I_{2j}$ such that $\calt_{2j} \lesst \bar{\calt}^{2j}$. By Fact \ref{full-extensions}, there is a $I_0$-full tower $\calt^{2j+1}$ indexed by $I_{\alpha_{2j+1}}$ such that $\calt^{2j+1} \upharpoonright I_{2j} = \bar{\calt}^{2j}$. In particular, $\calt_{2j} \lesst \calt_{2j+1}$. Now let $\alpha_{2j+2} = \alpha_{2j + 1}$, and take $\calt^{2j+2}$ to be any reduced tower with $\calt^{2j + 1} \lesst \calt^{2j + 2}$ by Fact \ref{reduced-extensions}. This completes the construction.

    By our construction specifications, Fact \ref{unions-of-reduced-are-reduced}, and Fact \ref{full-unions}, $\calt^\kappa$ is a reduced and $I_0$-full tower. By $I_0$-fullness, $M_{(s, t+1, 0)}^\kappa$ realises all the types over $M_{(s, t, 0)}$ for all $s<\kappa$ and $t < \lambda$, so $\langle M_{(s, 0, 0)}^\kappa : s < \kappa\rangle$ is $\lek^u$-increasing by Fact \ref{universal-and-limit-extensions-exist}. Since $\calt^\kappa$ is reduced, it is continuous at $(\kappa, 0, 0)$ (which has cofinality $\kappa$ in $I_{\alpha_\kappa}$) by Fact \ref{reduced-implies-high-continuity}. Therefore $M_{(0, 0, 0)} \underhlim M_{(\kappa, 0, 0)}^\kappa$.

    Since $M_{(0, 0, 0)}^\kappa$ and $M^*$ are $(\lambda, \geq \kappa)$-limit models, there is an isomorphism $f:M^* \cong M_{(0, 0, 0)}^\kappa$ by Fact \ref{bemain}. Take some $\bar{M}_1 \in \Kkappalims$ and $\bar{f}:M_1 \cong \bar{M}_1$ extending $f$. As $\bar{M}_1$ and $M^\kappa_{(\kappa, 0, 0)}$ are both $(\lambda, \geq \kappa)$-limit models over $M^\kappa_{(0, 0, 0)}$, by Fact \ref{bemain} again there exists an isomorphism $h:\bar{M}_1 \cong M^\kappa_{(\kappa, 0, 0)}$ fixing $M^\kappa_{(0, 0, 0)}$. Then $g = h \circ \bar{f}:M_1 \cong M_{(\kappa, 0, 0)}^\kappa$ is an isomorphism extending $f$. 

    Take any $i_0 \notin I_{\alpha_\kappa}$, and define $I = \{i_0\} \cup (\kappa \times \lambda \times \alpha_\kappa) \cup \{(\kappa, 0, 0)\}$, with the ordering placing $i_0$ as the minimal element, and the rest ordered lexicographically (note this is order isomorphic to the $I$ specified in the Proposition statement, with $\alpha = \alpha_\kappa$). Take $i_1 = (\kappa, 0, 0)$. Define $\calt = \langle N_i : i \in I \rangle ^\wedge \langle c_i : i \in I^-\rangle$ where
    \begin{enumerate}
        \item $N_{i_0} = M_0$
        \item for $i \geq_I (0, 0, 0)$, $N_i = g^{-1}[M_i^\kappa]$
        \item $c_{i_0} = b$
        \item for $i \geq_I (0, 0, 0)$,  $c_i = g^{-1}(a_i)$
    \end{enumerate}

    Note in particular that $N_{i_1} = M_1$ and $N_{(0, 0, 0)} = M^*$. 
    
    We have that $\calt \upharpoonright [i_0, (0, 0, 0)]_I = \langle M_0, M^*\rangle ^\wedge \langle b\rangle$ is reduced and $\calt \upharpoonright [(0, 0, 0), i_1]_I$ (being the image of $\calt^\kappa \upharpoonright [(0, 0, 0), (\kappa, 0, 0)]_{I_{\alpha_\kappa}}$ under $g^{-1}$) is reduced. So by Lemma \ref{reduced-stitching}, $\calt$ is reduced, as desired.
\end{proof}

\subsection{Main result and applications}\label{main-result-and-applications}

Now we restate and prove the main theorem.

\maintheorem*

\begin{proof}
	First recall that Hypothesis \ref{general-hypothesis} implies that Hypothesis \ref{proof-hypothesis} holds for $\dnf \upharpoonright \Kkappalims$. So assume Hypothesis \ref{proof-hypothesis} instead.
	
    Extending $M_1$ and $M_2$ if necessary, we may assume without loss of generality that $M_0 \underhlim M_1$ and $M_0 \underhlim M_2$. By Proposition \ref{reduced-tower-between-limit-models}, there exists a reduced tower $\calt^1 = \langle M_i^1 : i \leq \alpha \rangle ^\wedge \langle b_i : i < \alpha\rangle$ for some limit $\alpha$ such that $M_0^1 = M_0$, $M_\alpha^1 = M_1$, and $a_1 = b_0$. Let $\calt^2 = \langle M_0, M_2\rangle^\wedge \langle a_2\rangle$, which is strongly $\underhlim$-increasing. By Proposition \ref{tower-ap}, there exist $\langle M_{i, j} : i \leq \alpha, j < 2 \rangle$ and $f:M_2 \cong M_{0, 1}$ fixing $M_0$ such that 
    \begin{enumerate}
        \item $M_{i, 0} = M_i^1$ for $i \leq \alpha$
        \item for all $i < \alpha$ and $j < 2$, $\gtp(b_i/M_{i, j}, M_{i+1, j})$ $\dnf$-does not fork over $M_{i, 0}$
        \item for all $i \leq \alpha$, $\gtp(f(a_2)/M_{i, 0}, M_{i, 1})$ $\dnf$-does not fork over $M_{0, 0}$
        \item for all $i<k \leq \alpha$ and $j < 2$, $M_{i, j} \cap M_{k, 0} = M_{i, 0}$
    \end{enumerate}
    
    Take $N = M_{\alpha, 1}$. Then $M_1 = M_{\alpha, 0} \lek N$ and $f:M_2 \rightarrow N$. By condition (2) with $i = 0$ and $j=1$,  we have $\gtp(a_1/f[M_2], N) = \gtp(b_0/M_{0, 1}, M_{1, 1})$ $\dnf$-does not fork over $M_{0, 0} = M_0$. By condition (3) with $i = \alpha$, $\gtp(f(a_2)/M_1, N) = \gtp(f(a_2)/M_{\alpha, 0}, M_{\alpha, 1})$ $\dnf$-does not fork over $M_{0, 0} = M_0$. And by condition (4) with $i = 0, k = \alpha, j = 1$, we have that $M_1 \cap f[M_2] = M_{\alpha, 0} \cap M_{0, 1} = M_{0, 0} = M_0$. Thus taking $f_1:M_1 \rightarrow N$ to be the identity embedding and $f_2 = f$, $f_1$ and $f_2$ are as desired.

    The moreover part follows immediately from the previous statement, by replacing $M_1$ and $M_2$ by extensions in $\Kkappalims$, and taking $a_1\in M_1$ and $a_2\in M_2$ arbitrarily.
\end{proof}

\cite{bema} presents a number of settings in which Hypothesis \ref{general-hypothesis} holds in \cite[3.13]{bema}. We present the immediate application of our result in these settings now.

\begin{definition}
	Let $\K$ be an AEC stable in $\lambda \geq \LS(\K)$. Suppose $M \lek N$ and $p \in \gS(N)$. We say:
	\begin{enumerate}
		\item ({\cite[I.4.2]{van06}}) \emph{$p$ does not $\lambda$-split over $M$} if for all $N_1, N_2 \in \K$ where $M \lek N_l \lek N$ for $l = 1, 2$, and all $f : N_1 \cong N_2$ fixing $M$, we have $f(p \upharpoonright N_1) = p \upharpoonright N_2$.
		\item ({\cite[3.8]{vas16}}) \emph{$p$ does not $\lambda$-fork over $M$} if there exists $M_0 \lek^u M$ where $p$ does not $\lambda$-split over $M_0$.
		\item \emph{$\lambda$-non-forking} is the independence relation $\dnf$ where $a \dnf_{M_0}^N M$ if and only if $\gtp(a/M, N)$ does not $\lambda$-fork over $M_0$.
	\end{enumerate}
	
	If $M, N \in \K_{\geq \lambda}$, $M \lek N$, and $q \in \gS(N)$, we say
	\begin{enumerate}
		\item ({\cite[4.2]{vas16}}) \emph{$q$ does not $(\geq \lambda)$-fork over $M$} if there exists $M_0 \in \K_\lambda$ where $M_0 \lek M$ and for all $N_0 \in \K_\lambda$ where $M_0 \lek N_0 \lek N$, we have that $q \upharpoonright N_0$ does not $\lambda$-fork over $M_0$.
		\item \emph{$(\geq\lambda)$-non-forking} is the independence relation $\dnf$ where $a \dnf_{M_0}^N M$ if and only if $\gtp(a/M, N)$ does not $(\geq \lambda)$-fork over $M_0$.
	\end{enumerate}
\end{definition}

\begin{corollary}\label{superstable-corollary}
    Let $\K$ be a $\lambda$-superstable $\lambda$-symmetric AEC for some $\lambda \geq \LS(\K)$. Then $\lambda$-non-forking restricted to $\lambda$-limit models satisfies disjoint non-forking amalgamation. Moreover, every $\lambda$-limit model is a disjoint amalgamation base.
\end{corollary}

\begin{proof}
    By \cite[3.13(2)]{bema}, $\lambda$-non-forking on limit models satisfies Hypothesis \ref{proof-hypothesis} with $\kappa = \aleph_0$. The result is immediate from here by Theorem \ref{disjoint-nf-amalgamation}.
\end{proof}

Corollary \ref{superstable-corollary} could in fact could be proven using methods from \cite[\textsection 5]{vas19}, but our other examples rely on our generality of allowing $\kappa > \aleph_0$. Some terms in Corollary \ref{lrv-corollary} are introduced more formally in \cite[2.29]{bema}.

\begin{corollary}\label{lrv-corollary}
    Suppose $\K$ is an AEC, and $\dnf$ is a stable independence relation on $\K$ in the sense of \cite[\textsection 3]{lrv} such that $\dnfb{}{}{}{}$ satisfies $(\geq \kappa)$-local character and universal continuity. Then for any stable $\lambda \geq \LS(\K)$, $\dnfb{}{}{}{}$ restricted to $\Kkappalims$ satisfies disjoint non-forking amalgamation. In particular, every $M \in \Kkappalims$ is an amalgamation base for any stable $\lambda \geq \LS(\K)$. 
\end{corollary}

\begin{proof}
    Immediate from \cite[3.31(a)]{bema} and Theorem \ref{disjoint-nf-amalgamation}.
\end{proof}

\begin{corollary}\label{vasey-saturated-corollary}
    Suppose $\K$ is an AEC stable in $\mu \geq \LS(\K)$ and $\lambda \geq \mu^+$, where $\K$ has AP, NMM, $\mu$-tameness, and $\K_\lambda$ has JEP. Suppose also that $\mu$-non-splitting satisfies universal continuity on $\K$.
    
    Let $\dnf$ be $(\geq \mu)$-non-forking restricted to $\K_{(\lambda, \geq \mu^+)}$. Suppose $\dnf$ satisfies non-forking amalgamation. Then $\dnf$ satisfies disjoint non-forking amalgamation. In particular, every $(\lambda, \geq \mu^+)$-limit model is a disjoint amalgamation base.
\end{corollary}

\begin{proof}
	Immediate from \cite[3.41]{bema} and Theorem \ref{disjoint-nf-amalgamation} with $\kappa = \mu^+$.
\end{proof}

\subsection{Disjoint amalgamation from a weaker relation}\label{disjoint-amalgamation-from-a-weaker-relation}

In this section we describe how under weaker assumptions, we can use analogous machinery to get disjoint amalgamation in a broader class of strictly stable AECs, at the sacrifice of losing the non-forking of singletons. The proof is actually much shorter, though we will have to establish the new setting briefly.

More precisely, we assume the underlying framework of \cite{beard3}, which is identical to Hypothesis \ref{general-hypothesis}, but with the uniqueness, extension, and non-forking amalgamation assumptions replaced by their `weak' counterparts. First we introduce the weakened form of non-forking amalgamation, which is essentially the usual notion but where the non-forking is pushed down to a model $M_0$ which the base of the amalgamation $M$ is a particular kind of limit model over.

\begin{definition}\cite[Definition 2.23]{beard3}
	Let $\K$ be an AEC stable in $\lambda \geq \LS(\K)$, $\dnf$ an independence relation on an AC $\K'$ where $\Kkappalims \subseteq \K' \subseteq \K_\lambda$, and $\theta < \lambda^+$ a regular limit.
	
	We say $\dnf$ satisfies \emph{$(\lambda, \theta)$-weak non-forking amalgamation} if for every $M_0, M, M_1, M_2 \in \K'$ where $M$ is a $(\lambda, \gamma)$-limit model over $M_0$ and $M \lek M_l$ and $a_l \in M_l$ for $l = 1, 2$ such that $\gtp(a_l/M, M_l)$ $\dnf$-does not fork over $M_0$, then there exist $N \in \K'$ and $f_l : M_l \rightarrow N$ $\K$-embeddings fixing $M$ for $l = 1, 2$ such that $\gtp(f_l(a_l)/f_{3-l}[M_{3-l}, N)$ $\dnf$-does not fork over $M_0$, for $l = 1, 2$.
\end{definition}

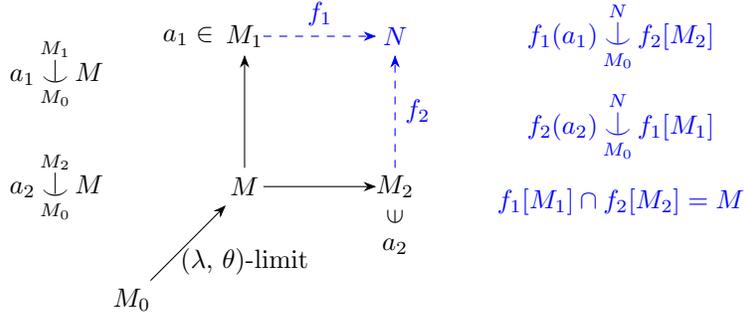
\begin{figure}[!ht]
\centering
\begin{circuitikz}
	\node [font=\normalsize] at (5.25,12) {$a_1 \in$};
	\node [font=\normalsize] at (8,9.2) {$a_2$};
	\node [font=\normalsize, rotate=90] at (8,9.6) {$\in$};
\node [font=\normalsize] at (3.5,11.5) {$a_1 \dnf_{M_0}^{M_1} M$};
\node [font=\normalsize] at (3.5,10) {$a_2 \dnf_{M_0}^{M_2} M$};
\node [font=\normalsize] at (4.5,8.5) {$M_0$};
\node [font=\normalsize] at (6,10) {$M$};
\draw [->, >=Stealth] (4.75,8.75) -- (5.75,9.75);
\draw [->, >=Stealth] (6,10.25) -- (6,11.75);
\draw [->, >=Stealth] (6.25,10) -- (7.75,10);
\node [font=\normalsize] at (8,10) {$M_2$};
\node [font=\normalsize] at (6,12) {$M_1$};
\node [font=\normalsize, color=blue] at (8.3,11) {$f_2$};
\node [font=\normalsize, color=blue] at (7,12.3) {$f_1$};
\draw [->, >=Stealth, dashed, color=blue] (6.25,12) -- (7.75,12);
\draw [->, >=Stealth, dashed, color=blue] (8,10.25) -- (8,11.75);
\node [font=\normalsize, color=blue] at (8,12) {$N$};
\node [font=\normalsize, color=blue] at (11,12) {$f_1(a_1) \dnf_{M_0}^N f_2[M_2]$};
\node [font=\normalsize, color=blue] at (11,10.8) {$f_2(a_2) \dnf_{M_0}^N f_1[M_1]$};
\node [font=\normalsize, color=blue] at (11,9.8) {$f_1[M_1] \cap f_2[M_2] = M$};
\node [font=\normalsize] at (6,9) {($\lambda$, $\theta$)-limit};
\end{circuitikz}

\caption{$(\lambda, \theta)$-weak non-forking amalgamation says that, given the solid black diagram, the dotted blue diagram exists.}
\label{diagram-weak-disjoint-non-forking-amalgamation}
\end{figure}

\begin{restatable}{theorem}{weak-theorem}\label{main-weak-theorem}
	Let $\K$ be an AEC stable in $\lambda$, where $\K_\lambda$ has AP. Let $\K'$ be an AC where $\Kkappalims \subseteq \K' \subseteq \K_\lambda$. Let $\dnf$ be an independence relation on $\K'$ satisfying weak uniqueness, weak existence, universal continuity* in $\K_\lambda$, $(\geq \kappa)$-local character, and $(\lambda, \theta)$-weak non-forking amalgamation in some regular $\theta \in [\kappa, \lambda^+)$.
	
	Then every $M \in \Kkappalims$ is a disjoint amalgamation base in $\K$.
\end{restatable}

For the rest of this subsection, assume the following, which is identical to \cite[Hypothesis 3.3]{beard3}:

\begin{hypothesis}\label{proof-weak-hypothesis}
	Let $\K$ be an AEC stable in $\lambda$, where $\K_\lambda$ has AP. Let $\dnf$ be an independence relation on $\Kkappalims$ satisfying weak uniqueness, weak existence, universal continuity* in $\K_\lambda$, $(\geq \kappa)$-local character, and $(\lambda, \theta)$-weak non-forking amalgamation in some regular $\theta \in [\kappa, \lambda^+)$.
\end{hypothesis}

Similar to before, we can build a theory of towers in this more general setting, though we will call them \emph{weak towers} to distinguish them from our other notion of tower from Defintion \ref{tower-definitions}. This theory is described in detail in \cite{beard3}, but we give an abridged overview here.

\begin{notation}
	Given a well order $I$, define the well order $I^{-2}$ as follows:
	\begin{enumerate}
		\item if $\operatorname{otp}(I)$ is $0$ or limit, $I^{-2} = I$
		\item if $\operatorname{otp}(I) = \gamma + 1$ for some limit $\gamma$ and $i_1$ is the final element of $I$, $I^{-2} = I \setminus \{i_1\}$ (with the restriction of the ordering of $I$).
		\item if $\operatorname{otp}(I) = \alpha + 2$ and $i_1$, $i_2$ are the last two elements of $I$, $I^{-2} = I \setminus \{i_1, i_2\}$ (with the restriction of the ordering of $I$).
	\end{enumerate}
	That is, $I^{-2}$ is $I$ with the final two elements removed, if they exist.
\end{notation}

We briefly recall the important notions surrounding weak towers.

\begin{definition}\label{weak towers}
	Assume Hypothesis \ref{proof-weak-hypothesis}. 
	
	A \emph{weak tower} is a sequence $\calt = \langle M_i : i \in I \rangle ^\wedge \langle a_i : i \in I^{-2} \rangle$ where $\langle M_i : i \in I \rangle$ is a $\lek$-increasing sequence in $\Kkappalims$, and $a_i \in M_{i+2}$ for all $i \in I^{-2}$.
	
	A weak tower is universal if and only if the sequence of models is $\lek^u$ increasing.
	
	Define the \emph{weak tower ordering} $\lesst^w$ on weak towers $\calt = \langle M_i : i \in I \rangle ^\wedge \langle a_i : i \in I^{-2} \rangle$ and $\calt' = \langle M_i' : i \in I' \rangle ^\wedge \langle a_i' : i \in (I')^{-2} \rangle$ by $\calt \lesst^w \calt'$ if and only if
	
	\begin{enumerate}
		\item $I \subseteq I'$
		\item for all $i \in I$, $M_i \lek^u M_i'$
		\item for all $i \in I^{-2}$, $a_i = a_i'$
		\item for all $i \in I^-$, $\gtp(a_i/M_i', M_{i+_I2}')$ $\dnf$-does not fork over $M_i$.
		\item For all $i \in I^-$, $i+_I1 = i+_{I'}1$.
	\end{enumerate}

	Given a $\lesst$-increasing sequence $\langle \calt^j : j < \alpha \rangle$ for limit $\alpha < \lambda^+$ where $\cof(\alpha) \geq \kappa$ and $\calt^j = \langle M_i^j : j \in I^j \rangle ^\wedge \langle a_i : i \in (I^j)^{-2} \rangle$, if $I^\alpha = \bigcup_{j < \alpha} I^j$ is a well ordering, \emph{union} of $\langle \calt^j : j < \alpha \rangle$ is the tower $\bigcup_{j < \alpha} \calt^j = \langle M_i^\alpha : j \in I^\alpha \rangle ^\wedge \langle a_i : i \in (I^\alpha)^{-2} \rangle$ where $M^\alpha_i = \bigcup_{j < \alpha} M^j_i$. This is a well defined weak tower.
	
	A \emph{chain of weak towers} is a $\lesst^w$-increasing sequence of weak towers $\langle \calt^j : j < \alpha \rangle$ with $\alpha \in [1, \lambda^+)$. The chain is \emph{brilliant} if for all $j < \alpha$, at least one of the following holds:
	\begin{enumerate}
		\item $\calt^j$ is universal
		\item $\cof(j) \geq \kappa$ and $\calt^j = \bigcup_{j' < j} \calt^{j'}$
	\end{enumerate} 
	
\end{definition}

\begin{remark}
	We have to consider chains as a whole rather than only individual towers because unlike $\lesst$, $\lesst^w$ is only known to be transitive on universal towers. By working with brilliant chains of weak towers, we avoid the issues that arise from this.
\end{remark}

The notion of a full weak tower is similar to before, but we have to consider reduced \emph{chains} of weak towers rather than single weak towers for technical reasons.

\begin{definition}
	Let $I_0 \subseteq I$ be well orderings. A weak tower $\calt = \langle M_i : i \in I \rangle ^\wedge \langle a_i : i \in I^{-2} \rangle$ is $I_0$-full if for all $i \in I_0$ and all $p \in \gS(M_i)$, there exists $k \in [i, i+_{I_0}1)_I$ and $M' \underhlim M_i$ such that $\gtp(a_k/M_{i+_I1}, M_{i+_I2})$ $\dnf$-does not fork over $M'$ and extends $p$.
\end{definition}

\begin{definition}
	Let $\calc = \langle \calt^j : j \leq \alpha\rangle$ be a brilliant chain of weak towers, where $\calt^\alpha = \langle M_i : i \in I \rangle ^\wedge \langle a_i : i \in I^{-2} \rangle$. We say $\calc$ is \emph{reduced} if for every $\calt^{\alpha+1} = \langle M_i' : i \in I \rangle ^\wedge \langle a_i : i \in I^{-2} \rangle$ such that $\langle \calt^j : j \leq \alpha + 1\rangle$ is brilliant, for all $r < s \in I$ we have $M_s \cap M_r' = M_r$.
\end{definition}

Sometimes we will only want to consider the sequence of an initial segment of towers. We formalise this below.

\begin{definition}
	Let $\calc = \langle \calt^j : j \leq \alpha\rangle$ be a reduced chain of weak towers, where $\calt^j = \langle M_i^j : i \in I^j \rangle ^\wedge \langle a_i : i \in (I^j)^{-2} \rangle$ for $j \leq \alpha$. Suppose $I_0$ is an initial segment of $I^\alpha$. Then $\calc \upharpoonright_* I_0 = \langle \calt^j \upharpoonright I_0 : j \leq \alpha \rangle$, where $\calt^j \upharpoonright I_0 =\langle M_i^j : i \in I^j \cap I_0 \rangle ^\wedge \langle a_i : i \in (I^j \cap I_0)^{-2} \rangle$.
\end{definition}

We use $\calc\upharpoonright_* I_0$ rather than $\calc \upharpoonright I_0$ since the latter could be misinterpreted as taking an initial segment of the sequence of towers, rather than the full sequence of initial segments of towers.

We summarise the important theory of weak towers below. This closely mirrors the theory of towers introduced in Subsection \ref{subsection-towers}, though often towers are replaced by brilliant chains of weak towers.

\begin{fact}\label{weak-fact}
	\begin{enumerate}
		\item\label{weak-fact-extend-any-chain-by-universal} \cite[Corollary 3.36]{beard3} If $\langle \calt^j : j < \alpha\rangle$ is a brilliant chain of towers where $\calt^j$ is indexed by $I^j$ and $I = \bigcup_{j<\gamma} I^j$ is a well order, then there exists $\calt^\alpha$ a universal tower indexed by $I$ such that $\langle \calt^j : j \leq \gamma\rangle$ is a brilliant chain of towers.
		
		\item\label{weak-fact-long-unions-can-be-appended} \cite[Lemma 3.31]{beard3} If $\langle \calt^j : j < \gamma\rangle$ is a brilliant chain of towers where $\calt^j$ is indexed by $I^j$, $\cof(\gamma) \geq \kappa$, and $\bigcup_{j<\gamma} I^j$ is a well order, then $\langle \calt^j : j \leq \gamma\rangle$ is a brilliant chain of towers.
		
		\item\label{weak-fact-can-extend-any-chain-to-reduced} \cite[Lemma 3.46]{beard3} For every brilliant chain of towers $\calc = \langle \calt^j : j \leq \alpha\rangle$ where $\calt^\alpha$ is indexed by $I$, there is a chain of towers $\hat{\calc} = \langle\hat{\calt}^j : j \leq \beta \rangle$ all indexed by $I$ such that $\calc^\wedge \hat{\calc}$ (the concatenation of the two chains) is brilliant and reduced.
		
		\item\label{weak-fact-unions-of-reduced-chains-are-reduced} \cite[Lemma 3.48]{beard3} Suppose $\langle\alpha_k : k < \gamma \rangle$ is an increasing sequence in $\lambda^+$. Let $\alpha = \bigcup_{k < \gamma} \alpha_k$. Suppose $\langle \calt^j : j < \alpha \rangle$ is a brilliant chain of towers such that $\langle \calt^j : j < \alpha_k \rangle$ is reduced for $k < \gamma$. Suppose $I^j$ indexes $\calt^j$, and that $\bigcup_{j < \alpha} I^j$ is a well order. Let $\calt^\alpha = \bigcup_{j < \alpha} \calt^j$. Then $\langle \calt^j : j \leq \alpha \rangle$ is a reduced chain of towers.
		
		\item\label{weak-fact-star-initial-segments-of-reduced-chains-are-reduced} \cite[Lemma 3.50]{beard3} If $\calc = \langle \calt^j : j \leq \alpha \rangle$ is a reduced sequence of towers where $\calt^\alpha$ is indexed by $I$, and $I_0$ is an initial segment of $I$, then $\calc \upharpoonright_* I_0$ is reduced.
		
		\item\label{weak-fact-final-towers-of-reduced-chains-are-continuous-at-high-cofinalities} \cite[Lemma 3.52]{beard3} If $\langle \calt^j : j \leq \alpha \rangle$ is brilliant and reduced where $\calt^\alpha = \langle M_i : i \in I \rangle ^\wedge \langle a_i : i \in I^{-2} \rangle$ and $\delta \in I$ with $\cof_I(\delta) \geq \kappa$, then $\bigcup_{i < \delta} M_i = M_\delta$.
		
		\item\label{weak-fact-full-extensions-exist} \cite[Lemma 3.44]{beard3} If $\alpha < \gamma < \lambda^+$ are limit ordinals with $\cof(\gamma) = \lambda$, $I$ is a well order, and $\calt$ is a universal weak tower indexed by $I \times \alpha$, then there is a universal and $(I \times \{0\})$-full weak tower $\calt'$ indexed by $I \times \gamma$ where $\calt \lesst \calt'$.
		
		\item\label{weak-fact-full-unions-are-full} \cite[Lemma 3.42]{beard3} Suppose $\langle\alpha_k : k < \gamma \rangle$ is an increasing sequence in $\lambda^+$ where $\cof(\gamma) \geq \kappa$. Let $\alpha = \bigcup_{k < \gamma} \alpha_k$. Suppose $\langle \calt^j: j < \alpha \rangle$ is a brilliant chain of towers where $\calt^j$ is indexed by $I^j$ for $j < \lambda^+$, and $\bigcup_{j < \alpha} I^j$ is a well order. Suppose $I_0 \subseteq I^0$. If $\calt^{\alpha_k}$ is $I_0$-full for all $k < \gamma$, then $\bigcup_{j < \gamma} \calt^j$ is $I_0$-full.
		
		\item\label{weak-fact-long-limits-are-isomorphic} \cite[Theorem 3.1]{beard3} Suppose $M, N_1, N_2 \in \K$ and $N_1, N_2$ are $(\lambda, \geq \kappa)$-limit models over $M$. Then $N_1 \underset{M}{\cong} N_2$.
		
		Moreover, if $N_1, N_2 \in \Kkappalims$, then $N_1 \cong N_2$ (even if they are not limits over the same base).
	\end{enumerate}
\end{fact}

Now we move towards proving Theorem \ref{main-weak-theorem}. The method is a simplified version of the approach used to prove Theorem \ref{disjoint-nf-amalgamation}: we will build a reduced chain with final weak tower $\calt = \langle N_i : i \leq \alpha \rangle ^\wedge \langle a_i : i < \alpha \rangle$ where $N_\alpha$ is a $(\lambda, \geq \kappa)$-limit model over the $N_0$, and take any universal weak tower $\calt' = \langle N_i' : i \leq \alpha \rangle ^\wedge \langle a_i : i < \alpha \rangle$ $\lesstrong^w$-extending the chain. The reduced condition gives that $N_\alpha \cap N_0' = N_0$. Then using uniqueness of limit models, we may `replace' $N_0$, $N_\alpha$ and $N_0'$ by $M_0$, $M_1$, and $M_2$ respectively to get the desired disjoint amalgamation.

First, we build the desired reduced chain.

\begin{lemma}\label{lemma-building-reduced-tower-over-a-limit-model}
	There exist $\alpha, \gamma < \lambda^+$ limit and a reduced chain $\langle \calt^j : j \leq \gamma \rangle$ with $\calt^\gamma = \langle M_i : i \leq \alpha \rangle$ such that $M_\alpha$ is a $(\lambda, \geq \kappa)$-limit model over $M_0$.
\end{lemma}

\begin{proof}
	The method is similar to Lemma \ref{building-a-reduced-tower-between-m0-m1}, but since we now must consider reduced chains rather than reduced towers, we must keep track of the chain built along the way (also similar to the final argument of \cite[Theorem 3.1]{beard3}). Note by relabelling the underlying well orders of the towers, it is enough to prove there exist a well order $I$ with initial element $i_0$ and final element $i_1$ where $i_1$ is limit, $\gamma < \lambda^+$ a limit ordinal, and a reduced chain $\langle \calt^j : j \leq \gamma \rangle$ with $\calt^\gamma = \langle M_i : i \in I \rangle$ such that $M_{i_1}$ is a $(\lambda, \geq \kappa)$-limit model over $M_{i_0}$.
	
	For $\beta \geq \omega$, let $I_\beta = (\kappa + 1) \times \lambda \times \delta$, and let $I_0 = (\kappa + 1) \times \lambda \times \{0\}$. By induction on $k \leq \kappa$, build continuous sequences of ordinals $\langle \alpha_k : k \leq \kappa \rangle$ and $\langle \beta_k : k \leq \kappa \rangle$, and a $\lesst^w$-increasing sequence of weak towers $\langle \calt^j : j \leq \beta_\kappa \rangle$ such that
	
	\begin{enumerate}
		\item For all $k \leq \kappa$ and $j \in [\beta_k, \beta_{k+1})$ (interpreting $\beta_{\kappa + 1}$ as $\beta_{\kappa}+1$), $\calt^j = \langle M_i^j : i \in I_{\alpha_k}\rangle^\wedge\langle a_i^j : i \in I_{\alpha_k}^-\rangle$
		\item $\alpha_0 = \omega$
		\item for all $k < \kappa$, $\beta_{k+1}$ is a successor and $\langle \calt^j : j < \beta_{k+1}\rangle$ is a reduced chain
		\item for all $k < \kappa$, $\calt^{\beta_{k+1}}$ is $(\kappa +1) \times \lambda \times \{0\}$-full
		\item $\calt^{\beta_\kappa} = \bigcup_{j < \beta_\kappa} \calt^j$
	\end{enumerate}
	
	This is possible via the following construction: suppose $k < \kappa$ and you have constructed $\alpha_l$, $\beta_{l}$, and $\langle \calt^j : j \leq \beta_l\rangle$ for all $l < k$.
	
	If $k = 0$, let $\alpha_0 = \omega$, $\beta_0 = 0$, and $\calt^0$ be any weak tower of index $I_\omega$. 
	
	If $k<\kappa$ is limit, take $\alpha_k = \bigcup_{l < k} \alpha_l$, $\beta_k = \bigcup_{l<k} \beta_l$, and $\calt^{\beta_k}$ any weak tower indexed by $I_{\beta_k}$ such that $\langle \calt^j : j \leq \beta^k\rangle$ is a brilliant chain of weak towers. 
	
	If $k$ is successor, note that $k = k^- + 1$. By Fact \ref{weak-fact}(\ref{weak-fact-can-extend-any-chain-to-reduced}) there exists a successor $\beta_k \in (\beta_{k^-}, \lambda^+)$ and weak towers $\calt^j$ for $j \in (\beta_{k^-}, \beta_k)$ indexed by $I_{\alpha_{k^-}}$ such that $\langle \calt^j : j < \beta_k \rangle$ is a reduced chain of weak towers. Take $\alpha_k = \alpha_{k^-} + \lambda$. By Fact \ref{weak-fact}(\ref{weak-fact-full-extensions-exist}), there is a $I_0$-full weak tower $\calt^{\beta_k}$ such that $\langle \calt^j : j \leq \beta_k\rangle$ is a brilliant chain of weak towers.
	
	If $k = \kappa$, let $\alpha_\kappa = \bigcup_{l < \kappa} \alpha_l$, $\beta_\kappa = \bigcup_{l<\kappa} \beta_l$, and $\calt^{\beta_\kappa} = \bigcup_{j < \beta_\kappa} \calt^j$. The conditions are preserved as $\cof(\kappa) \geq \kappa$. This completes the construction.
	
	Now we show this is enough. Note that conditions (3) and (6) of our construction guarantee that $\langle \calt^j : j \leq \beta_\kappa \rangle$ is reduced by Fact \ref{weak-fact}(\ref{weak-fact-unions-of-reduced-chains-are-reduced}). By conditions (4) and (6) of the construction, $\calt^{\beta_\kappa}$ is $I_0$-full.
	
	First note that as $\langle \calt^j : j \leq \beta_\kappa \rangle$ is reduced, and $\cof_{I_{\beta_\kappa}}((\kappa, 0, 0)) = \kappa$, by Fact \ref{weak-fact}(\ref{weak-fact-final-towers-of-reduced-chains-are-continuous-at-high-cofinalities}) $\calt^{\beta_\kappa}$ is continuous at $(\kappa, 0, 0)$, meaning $M^{\beta_\kappa}_{(\kappa, 0, 0)} = \bigcup_{i <_{I^{\beta_\kappa}} (\kappa, 0, 0)} M^{\beta_\kappa}_i$. Since $\calt^{\beta_\kappa}$ is $I_0$-full, for all $r < \kappa$ and $l < \lambda$, $M^{\beta_\kappa}_{(r, l+1, 0)}$ realises all types over $M^{\beta_\kappa}_{(r, l, 0)}$. Therefore by Fact \ref{universal-and-limit-extensions-exist}, for all $r < \kappa$, $M^{\beta_\kappa}_{(r+1, 0, 0)}$ is universal over $M^{\beta_\kappa}_{r, 0, 0)}$. So $M^{\beta_\kappa}_{(\kappa, 0, 0)} = \bigcup_{i <_{I_{\beta_\kappa}} (\kappa, 0, 0)} M^{\beta_\kappa}_i = \bigcup_{r < \kappa } M^{\beta_\kappa}_{(r, 0, 0)}$ is a $(\lambda, \kappa)$-limit model over $M^{\beta_\kappa}_{(0, 0, 0)}$. So taking $\hat{I} = [(0, 0, 0), (\kappa, 0, 0)]_{ I_{\beta_\kappa}}$ (that is, the initial segment of $I_{\beta_\kappa}$ up to and including $(\kappa, 0, 0)$), we have that $\langle \calt^j : j \leq \beta_\kappa \rangle \upharpoonright_* \hat{I}$ is a reduced chain by Fact \ref{weak-fact}(\ref{weak-fact-star-initial-segments-of-reduced-chains-are-reduced}). Thus replacing $I$ with $\hat{I}$, $\gamma$ with $\beta_\kappa$, $\langle \calt^j : j \leq \beta_\kappa \rangle$ with $\langle \calt^j : j \leq \beta_\kappa \rangle \upharpoonright_* \hat{I}$, $i_0$ with $(0, 0, 0)$, and $i_1$ with $(\kappa, 0, 0)$, we are done.
\end{proof}

\begin{proof}[Proof of Theorem \ref{main-weak-theorem}]
	
	First note that we can assume Hypothesis \ref{proof-weak-hypothesis} by restricting $\dnf$ to $\Kkappalims$ if necessary. We begin by building a single witness of disjoint amalgamation, which we will later use to prove the general case.
	
	By Lemma \ref{lemma-building-reduced-tower-over-a-limit-model}, there exist $\alpha, \gamma < \lambda^+$ limit and a reduced chain $\langle \calt^j : j \leq \gamma \rangle$ with $\calt^\gamma = \langle M_i : i \leq \alpha \rangle$ such that $M_\alpha$ is a $(\lambda, \geq \kappa)$-limit model over $M_0$. Take any universal weak tower $\calt^{\alpha+1} = \langle M_i' : i \leq \alpha \rangle$ such that $\langle \calt^j : j \leq \gamma+1 \rangle$ is brilliant by Fact \ref{weak-fact}(\ref{weak-fact-extend-any-chain-by-universal}). Since $\langle \calt^j : j \leq \gamma \rangle$ is reduced, $M_\alpha \cap M_0' = M_0$. So we have $M_0 \underhlim M_\alpha \lek M_\alpha'$ and $M_0 \lek^u M_0' \lek M_\alpha'$.
	
	Now we show every $N \in \Kkappalims$ is an amalgamation base in $\K$. Suppose $N_l \in \K_\lambda$ and $N \lek N_l$ for $l = 1, 2$. By Fact \ref{weak-fact}(\ref{weak-fact-long-limits-are-isomorphic}), there is an isomorphism $f : N \cong M_0$. Since $\calt^{\gamma} \lesst^w \calt^{\gamma+1}$, we have that $M_0 \lek^u M_\alpha$, so $f$ extends to an embedding $f_1 : N_1 \rightarrow M_\alpha$. Since $M_0 \lek^u M_0'$, there is an embedding $f_2 : N_2 \rightarrow M_0'$ extending $f$. As $M_\alpha, M_0' \lek M_\alpha'$, we may consider $f_1, f_2$ as embeddings into $N$.
	
	Since $M_0 \subseteq f_1[N_1] \subseteq M_\alpha$, $M_0 \subseteq f_2[N_2] \subseteq M_0'$, and $M_\alpha \cap M_0' = M_0$, we have $f_1[N_1] \cap f_2[N_2] = M_0 = f[N]$. By Lemma \ref{disjoint-ap-non-fixing-def}, this is enough.

\end{proof}

We have one immediate application of Theorem \ref{main-weak-theorem} that does not satisfy the hypotheses of Theorem \ref{disjoint-nf-amalgamation}.

\begin{corollary}
	Suppose $\K$ is an AEC stable in $\lambda \geq \LS(\K)$, that $\kappa < \lambda^+$ is regular, that $\theta \in [\kappa, \lambda^+)$ is regular. Suppose $\K_\lambda$ has AP, JEP, and NMM, and also that $\lambda$-non-splitting in $\K_\lambda$ satisfies continuity, $(\geq \kappa)$-local character, and $(\lambda, \theta)$-symmetry (see e.g. \cite[Definition 2.23]{beard3}).
	
	Then $\Kkappalims$ has disjoint non-forking amalgamation in $\K$.
\end{corollary}

\begin{proof}
	By the method of \cite[Corollary 4.1]{beard3}, $\lambda$-non-splitting satisfies Hypothesis \ref{proof-hypothesis} (which is identical to \cite[Hypothesis 3.3]{beard3}), and therefore Hypothesis \ref{general-hypothesis}. The result follows from Theorem \ref{main-weak-theorem}.
\end{proof}

As stated earlier, this result is stronger that Theorem \ref{disjoint-nf-amalgamation} in the sense that our relation has weaker assumptions, but weaker than Theorem \ref{disjoint-nf-amalgamation} in that we have only shown $(\lambda, \geq \kappa)$-limit models are disjoint amalgamation bases, without capturing any information about the non-forking properties of singletons. The following seems the most natural version of disjoint non-forking amalgamation in this `weaker' setting:

\begin{definition}
	Assume Hypothesis \ref{proof-weak-hypothesis}. We say $\dnf$ satisfies \emph{weak disjoint non-forking amalgamation} if for all $M_0, M, M_1, M_2 \in \Kkappalims$ such that $M$ is a $(\lambda, \geq \kappa)$-limit model over $M_0$, $M \lek M_l$ and $a_l \in M_l$ such that $\gtp(a_l/M, M_l)$ $\dnf$-does not fork over $M_0$, then there exist $N \in \Kkappalims$ and $f_l : M_l \rightarrow N$ fixing $M$ such that $\gtp(f_l(a_l)/M_{3-l}, N)$ $\dnf$-does not fork over $M_0$ for $l = 1, 2$, and $f_1[M_1] \cap f_2[M_2] = M$.
\end{definition}

\begin{remark}
	Note this is essentially $(\lambda, \theta)$-weak disjoint amalgamation, with the additional disjointness conclusion. Under Hypothesis \ref{proof-weak-hypothesis}, in light of Theorem \ref{weak-fact}(\ref{weak-fact-long-limits-are-isomorphic}), the definition would be equivalent if we added a $\theta$ parameter where $\cof(\theta) \geq \kappa$ and instead required $M_0$ be a $(\lambda, \theta)$-limit model over $M_*$.
\end{remark}

\begin{remark}
	Note given an independence relation with uniqueness and existence, this is a weaker property than disjoint non-forking amalgamation: monotonicity gives that the types $\dnf$-do not fork over $M_0$, and after using disjoint non-forking amalgamation, the conclusion follows from transitivity.
\end{remark}

This leads to the following natural conjecture:

\begin{question}\label{weak-question}
	Assume Hypothesis \ref{proof-weak-hypothesis}. Does $\dnf$ satisfy weak disjoint non-forking amalgamation?
\end{question}
	
We found that the proof of Theorem \ref{disjoint-nf-amalgamation} breaks down at certain points when using weak towers in place of towers and have thus been unable to answer Question \ref{weak-question}. 

Perhaps the most fundamental issue is that we would want the bottom the two towers being amalgamated to be equal to $M$ to get the correct intersection - but then we cannot capture the non-forking of the singletons over $M_0$ in the towers themselves. Weak transitivity seems insufficient to get around this complication, since the models in the reduced tower (the leftmost column) are not necessarily universal extensions of one another (in fact, since the tower is reduced, intuitively they will be very close to one another to guarantee the intersection property). With this in mind, perhaps a modified version of towers would be necessary to answer Question \ref{weak-question}.

\begin{remark}\label{remark-JEP-NMM-not-needed}
	In fact, both Theorem \ref{disjoint-nf-amalgamation} and Theorem \ref{main-weak-theorem} hold without the assumptions of JEP and NMM in $\K_\lambda$. If JEP fails, given $M \in \K_\lambda$, restrict to the AEC $\K_M$ of models which can be embedded into a model with $M$, as in \cite[Definition 3.12]{beard3}. All $M_1, M_2 \in \K$ with $M \lek M_1, M_2$ are in $\K_M$; the restriction of $\dnf$ to $\K_M$ retains its properties in $\K_M$; and $\K_M$ has JEP. So, we can prove the result in $\K_M$ instead. 
	
	Now we have JEP without loss of generality. Note that if NMM fails in $\K_\lambda$, then by JEP any $(\lambda, \geq \kappa)$-limit model is maximal. Thus all amalgamations are trivially disjoint, and the non-forking properties follow from existence. So we can assume NMM in $\K_\lambda$ holds without losing generality.
\end{remark}

\printbibliography

\end{document}